%% file: main.tex
\documentclass{article}
\usepackage{graphicx}
\usepackage[a4paper,margin=1in]{geometry}

\usepackage{authblk}

\makeatletter
\renewcommand\AB@affilsepx{ \protect\\ } 
\makeatother

\title{Simplicial Complex Emergence on Directed Hypergraphs}
\author[1,2,3]{Christian Kuehn}
\author[1]{Fergal Murphy\thanks{Email: fergal.murphy@tum.de}}
\affil[1]{Technical University of Munich, 
Department of Mathematics,  Munich, Germany}
\affil[2]{Munich Data Science Institute (MDSI),  Munich, Germany}
\affil[3]{Complexity Science Hub Vienna (CSH), Vienna, Austria}

\date{\today}

\input{packages}

\begin{document}
\maketitle

\begin{abstract}
    \noindent We study when co-evolving (or adaptive) higher-order networks defined on directed hypergraphs admit a simplicial description. Binary and triadic couplings are modelled by time-dependent weight tensors. Using  representation theory of the symmetric group $S_k$, we decompose these tensors into fully symmetric, fully antisymmetric, and mixed isotypic components, and track their Frobenius norms to define three asymptotic regimes and a quantitative notion of convergence. In the symmetric (resp. antisymmetric) limit, we certify emergence and stability of simplicial complexes via a local boundary test and interior drift conditions that enforce downward-closure; in the mixed limit, we show that the minimal faithful object is a semi-simplicial set. We illustrate the theory with simulations that track the isotypic Frobenius norms and the higher-order structure.  Practically, our work provides rigorous conditions under which homological tools are justified for adaptive higher-order systems.

\end{abstract}


\input{Introduction}

\input{section_2}

\input{Symmetric_regime}

\input{Antisymm_and_mixed}

\input{Conclusion}

\input{acknowledgments}

\printbibliography
\end{document}

%% file: packages.tex
\usepackage[T1]{fontenc}
\usepackage[utf8]{inputenc}
\usepackage[english]{babel}
\usepackage{microtype} 

\usepackage{amsmath, amssymb, amsthm, amsfonts} 
\usepackage{mathtools} 
\usepackage{bm, mathrsfs, enumitem, float}

\usepackage{graphicx, geometry}
\usepackage{hyperref}           
\usepackage[dvipsnames]{xcolor} 
\definecolor{dimgray}{rgb}{0.41, 0.41, 0.41} 
\usepackage{cleveref}           

\usepackage{booktabs}           
\usepackage{overpic}            
\usepackage{subcaption}         
\usepackage{tikz}
\usepackage{listings}
\usepackage{dirtytalk}
\usepackage{parskip}
\usepackage{csquotes}
\usepackage[maxbibnames=99]{biblatex}

\usetikzlibrary{calc, decorations.markings, arrows.meta, positioning, decorations.pathreplacing} 
\usetikzlibrary{cd} 

\newcommand{\R}{\mathbb{R}}

\DeclarePairedDelimiter{\norm}{\lVert}{\rVert} 

\usepackage[labelfont=bf]{caption} 
\renewenvironment{proof}{\emph{{{Proof :}}}}{\hfill \qedsymbol}
\numberwithin{equation}{section}

\theoremstyle{plain}
\newtheorem{theorem}{Theorem}[section]
\newtheorem{corollary}[theorem]{Corollary}

\theoremstyle{definition}
\newtheorem{definition}[theorem]{Definition}
\newtheorem{remark}[theorem]{Remark}
\newtheorem{example}[theorem]{Example}

\newcommand{\legendline}[2]{\color{#1}\rule[0.5ex]{4mm}{1mm}\color{black} \small #2}
\newcommand{\legenddashed}[2]{\color{#1}\rule[0.5ex]{1.5mm}{1mm} \rule[0.5ex]{0.5mm}{0mm} \rule[0.5ex]{1.5mm}{1mm}\color{black} \small #2}
\newcommand{\legenddotted}[2]{\color{#1}\rule[0.5ex]{0.5mm}{1mm} \rule[0.5ex]{0.5mm}{0mm} \rule[0.5ex]{0.5mm}{1mm} \rule[0.5ex]{0.5mm}{0mm} \rule[0.5ex]{0.5mm}{1mm}\color{black} \small #2}

\DeclareUnicodeCharacter{0308}{\relax}

%% file: Introduction.tex
\section{Introduction}

\subsection{Motivation and Overview}
Complex networks have been immensely successful in modelling a variety of systems; however, many real-world systems involve higher-order interactions that extend beyond pairwise couplings \cite{Battiston2020}, \cite{battiston2022higher}, \cite{bick2023higher}. Biological, social, and technological networks frequently exhibit group interactions, where collective effects emerge from triadic or even higher-order relationships. For example, in social systems, contagion and opinion dynamics often require group influence (so-called complex contagion) rather than just dyadic contacts  \cite{golovin2024polyadic}, \cite{Iacopini2019}, \cite{jardon2020fast} \cite{moussaid2013social}, \cite{monsted2017evidence}. Similarly, in neural systems, cognitive functions can involve the coordinated co-activation of groups of neurons, a phenomenon that cannot be captured by pairwise connections alone \cite{Petri2014,Reimann2017}. In response to the need to account for such multi-node dependencies, higher-order network models have been developed. One prominent example is the generalisation of the Kuramoto model to incorporate multi-body coupling: beyond the classical all-to-all sinusoidal coupling, higher-order Kuramoto models include terms coupling three or more oscillators at a time. These extensions lead to novel synchronisation patterns not present in pairwise network models, such as abrupt (explosive) transitions to synchronisation and multi-scale phase-locking behaviour \cite{delabays2019dynamical,dutta2023impact,gao2023dynamics,munyayev2023cyclops,xu2021spectrum}.

While hypergraphs provide a very general representation for any collection of multi-node interactions, simplicial complexes impose an important constraint reflecting many real systems: if a group interaction (simplex) exists, then all lower-order interactions among that group’s members also exist. In other words, a 2-simplex (triangle) in a simplicial complex implies its three edges are present, a 3-simplex (tetrahedron) implies all its triangles and edges are present, and so on. This “downward closure” property mirrors many real hierarchical structures. For instance, functional assemblies of neurons tend to have all pairwise connections among them, and in smaller social groups, if a triad interacts collectively, usually the constituent pairs also interact. By modelling higher-order interactions with simplicial complexes, one can enforce that group interactions are built atop their sub-interactions, which is both conceptually appealing and mathematically powerful. It enables the use of algebraic-topological tools to characterise structure and dynamics. For example, one can study the homology of simplicial complexes to identify high-dimensional holes or cycles in network data (invisible to edge-only analysis). Likewise, the Hodge Laplacian (a higher-order generalisation of the graph Laplacian) governs diffusion and synchronisation on $k$-simplices (groups of nodes), capturing dynamical modes involving loops and multi-node coordination beyond simple pairwise paths \cite{carletti2021random,carletti2020random,jost2021normalized,mulas2021spectral,mulas2022random}. Recent studies have shown that the presence of a simplicial complex structure (as opposed to an arbitrary hypergraph) can qualitatively alter collective dynamics. For instance, higher-order interactions may enhance or impede synchronisation in markedly different ways depending on whether the network has the simplicial inclusion property or not. These insights highlight that modelling choices at the higher-order level (hypergraph vs. simplicial complex) profoundly influence system behaviour, providing one motivation for our focus on when and how simplicial complex structures can emerge in adaptive hypergraph dynamics \cite{millan2020explosive,zhang2023higher}.

Indeed, another critical aspect of many complex systems is adaptivity: the network topology itself co-evolves with the node dynamics. In an adaptive network, links can appear, disappear, or change weight based on the state of the nodes, creating a feedback loop between structure and dynamics. This co-evolution is widely observed. In epidemics, susceptible individuals rewire contacts to avoid infection, thereby rewiring the network as the disease spreads \cite{Gross2006}. In neural circuits, synaptic connections strengthen or weaken depending on activity (Hebbian plasticity and other forms of learning) \cite{morrison2008phenomenological}. In infrastructure and transportation networks, heavily used routes are built or reinforced while under-utilised connections decay \cite{xie2007modeling}. Adaptive network models capture this by allowing the graph’s edges to change as a function of node states, which can yield rich phenomena such as the formation of robust network communities, oscillatory dynamics, or sudden tipping points in network topology, \cite{bdolach2025tipping,berner2023adaptive,gross2008adaptive,ha2016synchronization}.
 However, most existing adaptive network models have considered only pairwise interactions, focusing on the creation or deletion of links between individual nodes. Classic examples include adaptive epidemics on networks, where links rewire based on infection status, or adaptive consensus and oscillator networks, where connection weights adjust according to synchronisation error, \cite{berner2023adaptive}. These studies demonstrated that adaptivity can produce assortative link patterns, multi-stability, and abrupt transitions even in pairwise networks. Yet they neglect the possibility that higher-order group interactions (hyperedges) might also form or dissolve in response to the system’s state. In reality, structural changes can occur at multiple scales: entire group interactions (a meeting, a team, a synchronised neural ensemble) may appear or break apart depending on the dynamics. For example, during a contagion event, not only might individuals sever contacts, but group gatherings could be cancelled, removing higher-order contagion pathways beyond pairwise contacts. Likewise, in a learning brain, a cohort of neurons might wire together or disband as a group. Ignoring these higher-order topological adaptations leaves out an important dimension of co-evolution. By studying simplicial complex formation in an adaptive setting, we aim to capture how both individual links and multi-node interaction patterns reorganise in concert with nodal dynamics. This provides a more comprehensive view of phenomena like epidemics (where avoidance of group events is as important as breaking individual contacts) and neural learning (where ensembles of neurons rewire en masse). 
 
 In the following, we build a modelling framework to investigate adaptive higher-order networks and identify the conditions under which their self-organised dynamics give rise to an emergent simplicial complex structure. For example, in works such as  \cite{anwar2024synchronization,dutta2024transition,skardal2020higher}, a unifying goal is to study Kuramoto-type dynamics on assumed higher-order structures such as simplicial complexes. It is the chief aim of this work to give a structural and dynamical foundation for when these higher-order simplicial structures exist at all.
 
\subsection{Outline and Summary of Main Results}

Our aim is to unify three axes that are often treated separately: higher-order interactions, adaptivity of the coupling structure, and rigorous simplicial complex emergence. We adopt a tensor-first viewpoint on directed hypergraphs and use representation theory of the symmetric group to track how symmetry and adaptivity co-produce combinatorial structure.  

The main contributions of this work are as follows:
\begin{itemize}
  \item We define adaptive triadic dynamical systems $(\textbf{x},A^{(1)},A^{(2)})$ and introduce directed hypergraphs as tensorial objects, together with a representation-theoretic splitting of adjacency tensors into symmetric, antisymmetric, and mixed components. 

  \item By monitoring the Frobenius norms of these components, we identify three long-term symmetry regimes and show that each determines a distinct class of emergent simplicial object. 
  \item \textbf{Main Theorems.} We prove three main emergence results, which can be summarised informally as follows:
  \begin{enumerate}
    \item \emph{Theorem \ref{thm:retention} (Symmetric regime):} once the system enters the symmetric regime, it retains an unoriented simplicial complex structure for all future time.  
    \item \emph{Theorem \ref{thm: oriented retention theorem} (Antisymmetric regime):} once in the antisymmetric regime, it retains an oriented simplicial complex structure.  
    \item \emph{Theorem \ref{thm: mixed regime retention main theorem} (Mixed regime):} in the absence of symmetry, the minimal faithful combinatorial object is a semi-simplicial set, and retention holds under analogous conditions.  
  \end{enumerate}

\end{itemize}
The proofs combine local boundary tests with interior drift estimates to enforce downward closure, providing sufficient conditions under which the emergent simplicial structure remains invariant.  
 Moreover, by using Frobenius-norm order parameters tied to representation-theoretic components, we provide a quantitative notion of convergence toward each regime. Numerical simulations are used to practically demonstrate and cross-validate the analytic predictions. In the numerics we also track the evolution of the symmetry norms across regimes. Taken together, these results give the first rigorous demonstration of emergent simplicial structure in adaptive higher-order network dynamics.  

Section 2 introduces the model and tensor decomposition. Section 3 develops the symmetric theory and illustrates with examples.  Section 4 studies the antisymmetric and mixed regime. Section 5 concludes with a discussion and outlook.

%% file: section_2.tex
\section{Asymptotic Regimes of Coevolving Hypergraphs.}

\subsection{Background and Canonical Model}

The general structure of a coevolving network dynamical system with $n \in \mathbb{N}$ particles on a hypergraph with higher-order interactions can be described as follows. Each particle $i \in \{1, \dots , N\}$ carries a dynamical state $x_i(t) \in \mathcal{M},$ where $\mathcal{M}$ is a smooth manifold (often $\mathcal{M} = \mathbb{S}^1$ in Kuramoto-type models, or $\mathcal{M} = \mathbb{R}^d$ in Euclidean cases). Each particle also has an intrinsic parameter $\omega_i \in \Omega \subseteq \mathbb{R}^d$; in the most common case $d = 1$, so $\omega_i$ is simply a scalar frequency. For every order $m \geq 1$ we introduce a time-dependent weight tensor 
$$
A^{(m)}(t)=\bigl(A^{(m)}_{i_0\dots i_m}(t)\bigr)_{1\le i_0,\dots ,i_m\le N}\in
\bigl(\mathbb R^{N}\bigr)^{\otimes(m+1)},
$$
with no a-priori symmetry, where $t\in\R$ denotes time. Let $\mathbf{x} = (x_1, \dots , x_N)^\intercal$ . Such systems are governed by ordinary differential equations (ODEs) on the nodes and tensor entries as follows, 
$$
\dot x_i = f_i\bigl(\omega_i, \textbf{x}, \{A^{(m)}\}_{m\ge1}\bigr),\qquad
\dot A^{(m)}_{i_0\dots i_m}=g^{(m)}_{i_0\dots i_m}\bigl(A^{(m)},\mathbf x\bigr),
$$
where overdot denotes the time derivative, $f_i$ are vector fields for the phase space dynamics of each node, and the functions $g_m $ describe the evolution of the network topology. Throughout, we will mainly deal with systems where $1 \leq m \leq 2$. Thus, the dynamical systems of main interest to us are defined on a network with higher-order interactions encoded in the matrix $A^{(1)}$ and tensor $A^{(2)}.$ Since these tensors change dynamically, the underlying hypergraph will therefore also have a coevolving topology. \\

\begin{definition}[Adaptive, Triadic, Network Dynamical System]\label{def: adaptive triadic NDS}
    By an adaptive, triadic network dynamical system, we mean a system as above but with $1 \leq m \leq 2.$ That is, there are  $N \geq 3$ nodes on an evolving hypergraph, with dynamical states given by  $x_i \in \mathcal{M}$, where $\mathcal{M}$ is a smooth manifold, and hypergraph binary and triadic connections encoded in a weighted matrix $A^{(1)} \in \mathbb{R}^{N \times N}$ and a tensor $A^{(2)} \in \left( \mathbb{R}^N\right)^{\otimes 3}$, governed by equations 
$$\dot x_i = f_i\bigl(\omega_i, \textbf{x}, A^{(1)} , A^{(2)}\bigr),\qquad
\dot A^{(1)}_{ij}=g^{(1)}_{ij}\bigl(A^{(1)},\mathbf x\bigr), \qquad \dot A^{(2)}_{ijk}=g^{(2)}_{ijk}\bigl(A^{(2)},\mathbf x\bigr),$$
where each $\omega_i \in \mathbb{R}$ is an intrinsic parameter depending on $x_i.$
We denote such an adaptive, triadic, network dynamical system by  a triple $\mathcal{H}(t) = (\textbf{x}, A^{(1)}, A^{(2)} ), $ where $\textbf{x} \in \mathbb{R}^N$ are the node variables, and $A^{(1)} \in \mathbb{R}^{N \times N}$, $A^{(2)} \in \left( \mathbb{R}^N \right)^{\otimes 3}$ are the weighted adjacency matrix and tensor respectively.\\

\end{definition}

\begin{remark}
    Unless stated otherwise, we assume that the initial data $
(\textbf{x}(0),A^{(1)}(0),A^{(2)}(0))$
is randomly distributed according to a probability measure $\mu$ on 
$\mathcal{X} := \mathbb{R}^N \times (\mathbb{R}^{N})^{\otimes 2} \times (\mathbb{R}^{N})^{\otimes 3}$ 
which is absolutely continuous with respect to Lebesgue measure. 
Equivalently, $\mu$ admits a Radon-Nikodym derivative  $f \in L^1$, 
so that $\textnormal{d}\mu = f \, \textnormal{d} \lambda$ with $\lambda$ the product Lebesgue measure. 
The standard example is the Gaussian distribution, or uniform distribution on bounded subsets of $\mathcal{X}$. \\
Furthermore, in general, we restrict the adaptive laws $g^{(1)}_{ij}$ and $g^{(2)}_{ijk}$ to depend only on local variables, i.e. the states and interactions associated with the participating nodes $(i,j)$ or $(i,j,k).$
\end{remark}

The prototypical example of such an adaptive triadic network dynamical system is given by a variant of the Kuramoto model, and is ubiquitous in our study since it is the easiest non-trivial, adaptive, triadic higher-order network dynamical system to define. \\

\begin{example}[Generalised, Adaptive, Triadic Higher–Order Kuramoto Model]\label{def:ATKHH}
Let $N\in\mathbb{N}$ and let $\theta=\theta_i(t)\in \mathbb{S}^1$ denote the position of particle $i$. The dynamics of the particles and the time–evolving binary and triadic coupling tensors are given by

$$
\left\{
\begin{aligned}
\dot{\theta}_i \;&=\; \omega_i
+ \frac{1}{N}\sum_{j=1}^N A^{(1)}_{ij}\, f_1(\theta_i,\theta_j)
+ \frac{1}{N^2}\sum_{j=1}^N \sum_{k=1}^N A^{(2)}_{ijk}\, f_2(\theta_i,\theta_j,\theta_k), \\[0.4em]
\dot{A}^{(1)}_{ij} \;&=\; g^{(1)}\big(A^{(1)},A^{(2)},\theta\big), \\[0.2em]
\dot{A}^{(2)}_{ijk} \;&=\; g^{(2)}\big(A^{(1)},A^{(2)},\theta\big),
\end{aligned}
\right.
$$

where $\theta=(\theta_1,\dots,\theta_N)\in 
\left( \mathbb{S}^1 \right)^N$ and:
\begin{itemize}
  \item $\omega_i\in\mathbb{R}$ is the intrinsic frequency of oscillator $i$;
  \item $f_1: \mathbb{R}^2 \to\mathbb{R}$ and $f_2: \mathbb{R}^3 \to \mathbb{R}$ are coupling functions;
  \item $A^{(1)}(t)=\big(A^{(1)}_{ij}(t)\big)_{1\leq i,j\leq N}\in\mathbb{R}^{N\times N}  \cong \left( \mathbb{R}^N \right)^{\otimes 2}.$ is the time–evolving weighted adjacency matrix;
  \item $A^{(2)}(t)=\big(A^{(2)}_{ijk}(t)\big)_{1\leq i,j,k\leq N}\in\mathbb{R}^{N\times N\times N} \cong \left(\mathbb{R}^N\right)^{\otimes 3}$ is the time–evolving rank-$3$ tensor of triadic weights;
  \item $g^{(1)}$ and $g^{(2)}$ are the componentwise adaptation maps on the binary and triadic connections. \\
\end{itemize}

If we assume $A^{(2)}=0$ and that $A^{(1)}_{ij}=K$ for all $i\neq j$, where $K>0$ is a constant
coupling parameter, then the model reduces to the classical Kuramoto system with all-to-all
interactions,
$$
\dot{\theta}_i = \omega_i + \frac{K}{N}\sum_{j=1}^N f_1(\theta_i,\theta_j).
$$
In the standard case $f_1(\theta_i,\theta_j)=\sin(\theta_j-\theta_i)$ this is precisely the
well-studied Kuramoto model. To quantify the degree of collective synchrony among the oscillators, the complex order parameter  is used, given by

$$r e^{i \psi} := \frac{1}{N}\sum_{j = 1}^Ne^{i \theta_j}.$$
Here, $r$ represents the magnitude of the macroscopic coherence ($0 \leq r \leq 1)$ and $\psi$ is the average phase of the population of oscillators. If $r \approx 0$, the phases are uniformly distributed corresponding to an incoherent or desynchronised state, and if $r = 1$, the oscillators are perfectly in phase, representing full synchronisation.
As $N \to \infty$, one may describe the system in terms of the continuum distribution 
of oscillator phases, leading to a nonlinear McKean–Vlasov type (or Vlasov/Fokker–Planck) 
equation for the probability density of phases 
\cite{acebron2005kuramoto, strogatz2000kuramoto,}. 
In this mean–field formulation, Kuramoto conjectured that for unimodal, symmetric 
frequency distributions $g(\omega)$, there exists a sharp critical coupling
\[
  K_c \;=\; \frac{2}{\pi g(0)},
\]
such that the incoherent (uniform) state is stable for $K<K_c$, while for $K>K_c$ 
a branch of partially synchronised states bifurcates. 
This conjecture was rigorously established by Chiba 
\cite{chiba2015proof}, who applied center manifold reduction to the continuum PDE, 
thereby providing the first proof of the Kuramoto transition in the large $N$ 
limit.\\

\end{example}

\begin{remark}[Directed hypergraph conventions]
In the combinatorial and computer science literature, a directed hypergraph is usually 
defined in a set--theoretic fashion: each hyperedge is a pair $(T,H)$, where $T$ is a set of 
tail (input) nodes and $H$ is a set of head (output) nodes, with possible variants that 
also allow auxiliary or intermediate nodes. This convention is well-suited to problems in logic, 
database theory, and transportation analysis, where the distinction between inputs and outputs 
is fundamental, \cite{gallo1993directed}. 

There are other, differing definitions of what it means for a hypergraph to be directed, usually depending on the application considered. For instance, in \cite{DUCOURNAU201491} a directed hypergraph is taken to mean an ordered pair 
$
\overrightarrow{H} = (V, \overrightarrow{E}),
$
where $V$ is a finite set of vertices and $\overrightarrow{E} = \{\overrightarrow{e}_i : i \in I\}$ is a set of hyperarcs indexed by a finite set $I = \{1,\dots,M\}$, with $M = |\overrightarrow{E}|$. Each hyperarc $\overrightarrow{e}_i$ is written as
$
\overrightarrow{e}_i = (e_i^+, e_i^-),
$
where $e_i^+ \subseteq V$ is the tail of the hyperarc, and $e_i^- \subseteq V$ is the head of the hyperarc. Moreover, the authors require that 
$ e_i^+ \neq \emptyset, e_i^- \neq \emptyset,$ and $ e_i^+ \cap e_i^- = \emptyset.$ The vertices of a hyperarc are $e_i = e_i^+ \cup e_i^-$.   One can recover the corresponding undirected hypergraph by simply taking $H = (V,E)$ with $E = \{ e_i \mid i \in I \}.$ It is shown in \cite{DUCOURNAU201491} that a formulation of random walks on these directed hypergraphs can be introduced that is relevant to image processing.  One sees that this definition of directed hypergraph and the traditional set-theoretic one from \cite{gallo1993directed} are closely related but not identical.  

Given a triadic, adaptive dynamical system  $\mathcal{H} = (\textbf{x},A^{(1)},A^{(2)})$ as in Definition \ref{def: adaptive triadic NDS}, it is not exactly clear how one can recover a directed Hypergraph one commonly finds in the literature, such as the ones detailed above. The problem lies in the mismatch between our tensor based formalism, and the set-theoretic formalism commonly used in the existing hypergraph literature.  Indeed, this lack of convertibility between the tensor based approach and a set-theoretic description was also discussed in \cite{contreras2024beyond}, where, motivated by eigenvector centrality equations,  the authors introduce the notion of \emph{heterogeneous} hypergraphs. These are hypergraphs that are defined algebraically, that is, through a tensor. Importantly, they contain the set-theoretic definitions above as a special cases under certain symmetry requirements. Since our adaptive network dynamical systems $(\textbf{x},A^{(1)},A^{(2)})$ are tensors, which we want to treat as directed hypergraphs and study simplicial complex emergence on, we therefore make the following definition that will be used throughout this paper. \\

\end{remark}

\begin{definition}[Directed hypergraph; algebraic form]
Let $N\in\mathbb{N}$ and $V=\{1,\dots,N\}$.  
A \emph{directed hypergraph} on $V$ is specified by a collection of adjacency tensors
$$
A^{(m)} \in (\mathbb{R}^{N})^{\otimes m+1}, \qquad m=1,2,3,\dots,
$$
where the entry $A^{(m)}_{i_1 \dots i_m}$ represents the weight of an $m$--ary hyperedge between the vertices $(i_1,\dots,i_m)$.  

No symmetry is assumed on the indices of $A^{(m)}$.  
\begin{itemize}
    \item If $A^{(m)}$ is invariant under all index permutations, then the hyperedges are \emph{undirected}.  
    \item If $A^{(m)}$ is not fully symmetric, we call the resulting hyperedges \emph{directed}, the directedness being encoded by which permutations preserve or break equality of tensor entries.  \\
\end{itemize}

\end{definition}

\begin{remark}[Adaptive triadic systems as algebraic directed hypergraphs]
The adaptive triadic network dynamical system $\mathcal{H}(t)=(\textbf{x},A^{(1)},A^{(2)})$ of Definition \ref{def: adaptive triadic NDS}
naturally gives rise to a directed hypergraph in the algebraic sense.  
At each time $t$, the weighted adjacency matrix $A^{(1)}(t)\in \mathbb{R}^{N \times N}$ and the triadic adjacency tensor 
$A^{(2)}(t)\in (\mathbb{R}^N)^{\otimes 3}$ specify the collection of binary and ternary hyperedges through their tensor entries.  

In this framework, a directed hypergraph is nothing more than the data of these tensors, with \emph{directedness} encoded
by the lack of full permutation symmetry in their indices:
\begin{itemize}
    \item if $A^{(m)}$ is invariant under all permutations of indices, the corresponding hyperedges are undirected;
    \item if $A^{(m)}$ is not fully symmetric, the hyperedges are directed, with the asymmetry specifying how roles are distinguished among the vertices.
\end{itemize}
Thus our adaptive dynamics do not produce just one fixed type of directed hypergraph (such as the classical head--tail version),
but a whole family of algebraically defined directed hypergraphs whose structure can evolve over time.  
In particular, depending on which symmetry class the tensors flow into, one recovers unoriented simplicial complexes (symmetric regime), oriented simplicial complexes (antisymmetric regime), or more general simplicial sets (mixed regime).  \\
\end{remark}

\color{black}

\begin{example}

Suppose we have an adaptive coevolving network dynamical system with the following dynamics on the edges. 

$$\dot A^{(1)}_{ij} = \begin{cases}
    -\delta_{1}\bigl(A^{(1)}_{ij} - \cos(\theta_{i}-\theta_{j})\bigr) & i \neq j\\
    0 & \text{otherwise}
\end{cases}$$

where $\delta_1 > 0$ is some parameter. Let $\Delta_{ij}$ be the difference between $A^{(1)}_{ij}$ and $A^{(1)}_{ji},$ i.e.  \break $ 
  \Delta_{ij}(t) \;:=\; A^{(1)}_{ij}(t) - A^{(1)}_{ji}(t).$
Differentiating and using the fact that
$\cos(\theta_{i}-\theta_{j}) = \cos(\theta_{j}-\theta_{i})$ gives
$$
  \dot\Delta_{ij}
  = \dot A^{(1)}_{ij} \;-\; \dot A^{(1)}_{ji}
  = -\delta_{1}\bigl(A^{(1)}_{ij}-\cos(\theta_{i}-\theta_{j})\bigr)
    +\delta_{1}\bigl(A^{(1)}_{ji}-\cos(\theta_{j}-\theta_{i})\bigr)
  = -\delta_{1}\,\Delta_{ij}.
$$
Hence $  \Delta_{ij}(t) = \Delta_{ij}(0)\,e^{-\delta_{1}t}, $ and so the difference decays to zero. In this case, the best we can say is that $A^{(1)}_{ij}$ and $A^{(1)}_{ji}$ converge to one another. However, unless the initial conditions are the same, they will not be exactly equal for all time $t \geq 0.$\\

\end{example}

\begin{remark}
    The above example shows that on the level of binary interactions, the directed graph will enter an eventually symmetric regime, that is, one in which the difference between $A^{(1)}_{ij}$ and $A^{(1)}_{ji}$ becomes negligible in the limit. If one replaces $\cos(\theta_i - \theta_j)$ above with a $\sin(\theta_i - \theta_j)$, then due to $\sin$ being an odd function, one can see that as $t \rightarrow \infty$, we would have $A_{ij}^{(1)}(t) \rightarrow - A_{ji}^{(1)}(t).$ The graph would in this case enter an antisymmetric regime. If we were simply interested in the formation of $1$-simplices and thereby only looking at the directed graph structure, one sees from this example the importance of entering either an eventually \emph{symmetric} or \emph{antisymmetric} regime. In the former, one can study standard (unoriented) simplicial complex emergence; in the latter, one can study oriented simplicial complex emergence.
\end{remark}

The higher–order case amplifies this point: for $k\geq 3$ there is no canonical input–output labelling, so conclusions must be invariant under permutations of tensor slots. This naturally leads us to organise $A^{(k)}$ by the $S_{k+1}$–symmetries of its indices.
In the next section, we set up the $S_k$ action on our directed hypergraphs and introduce the canonical components that will serve as relabelling–invariant order parameters for the subsequent convergence and simplicial representation results.

\subsection{Tensor Spaces and Symmetric Group Actions}

We begin by formally defining the spaces that contain our edge and hyperedge data, and the natural actions of symmetric groups on them.

The set of all possible edges between $N$ vertices is described by a matrix 
$$
A^{(1)} = \left( A^{(1)}_{ij} \right)_{1 \leq i,j \leq N } \in\mathbb{R}^{N \times N}.
$$
This space of $N \times N$ real matrices is isomorphic to the tensor product of the vector space $\mathbb{R}^N$ with itself, which we denote as $V_2$:
$$
V_2 := \mathbb{R}^{N \times N} \cong (\mathbb{R}^N)^{\otimes 2}.
$$
Similarly, the set of all possible triangular hyperedges (or 2-simplices) is described by a rank-3 tensor
$$
A^{(2)} = \left( A^{(2)}_{ijk} \right)_{1 \leq i,j,k \leq N },
$$
which resides in the space we denote by $V_3$:
$$
V_3 := (\mathbb{R}^N)^{\otimes 3}.
$$

For the space $V_2$, the indices $(i, j)$ can be permuted by the symmetric group on two elements, $S_2 = \{e, (12)\}$, where $e$ is the identity and $(12)$ is the transposition swapping 1 and 2. We define a map $\rho^2: S_2 \rightarrow GL(V_2)$ that describes how a permutation $\pi \in S_2$ acts on a matrix $A^{(1)} \in V_2$. For a given $\pi$, the action on $A^{(1)}$ produces a new matrix whose $(i,j)$-th entry is determined by permuting the indices of $A^{(1)}$ according to $\pi^{-1}$:
$$
\rho^2(\pi) \left( A^{(1)} \right)_{ij} := A^{(1)}_{\pi^{-1}(i) \pi^{-1}(j)}.
$$
This map $\rho^2$ is a \emph{group representation}, which is a homomorphism from $S_2$ to the general linear group of $V_2$. The use of $\pi^{-1}$ is a standard convention to ensure that we define a left action.

Analogously, for the space $V_3$, the indices $(i, j, k)$ can be permuted by the symmetric group on three elements, $S_3$. We define a representation $\rho^3: S_3 \rightarrow GL(V_3)$ in the exact same manner:
$$
\rho^3(\pi) \left( A^{(2)} \right)_{ijk} := A^{(2)}_{\pi^{-1}(i) \pi^{-1}(j) \pi^{-1}(k)}.
$$

\begin{definition}[Representation and Irreducibility]
A \emph{representation} of a group $G$ on a vector space $V$ is a group homomorphism $\rho: G \rightarrow GL(V)$. A subspace $W \subseteq V$ is called a \emph{subrepresentation} if it is invariant under the action of $G$, i.e., $\rho(g)w \in W$ for all $g \in G$ and $w \in W$. A representation $V$ is said to be \emph{irreducible} if its only subrepresentations are $\{0\}$ and $V$ itself.
\end{definition}

From the representation theory of $S_k$ and Schur-Weyl duality, it is known $V_k$ decomposes into a direct sum of irreducible representations of $S_k,$\cite{fulton2013representation}, $\S 6.1.$ 
For $V_2$, the decomposition (up to isomorphic copies - also known as the isotypic decomposition), is given by
$$
V_2 = \mathrm{Sym}^2(\mathbb{R}^N) \oplus \Lambda^2 \mathbb{R}^N, 
$$
with idempotent projections given by 
$$
A_{\mathrm{sym}}:=P_{\mathrm{sym}}A
      =\tfrac12\bigl(A+A^{\!\top}\bigr)\in \text{Sym}^2(\mathbb{R}^N), \qquad
A_{\mathrm{alt}}:=P_{\mathrm{alt}}A
      =\tfrac12\bigl(A-A^{\!\top}\bigr)\in \Lambda^{2}(\mathbb{R}^N).
$$
We see that this decomposes a matrix into a purely symmetric component and a purely skew-symmetric component. Similarly, for $V_3$ the decomposition is given by

$$
V_{3}
   \;=\;
   \underbrace{\text{Sym}^{3}(\mathbb{R}^{N})}_{\text{fully symmetric}}
   \;\oplus\;
   \underbrace{\Lambda^{3}(\mathbb{R}^{N})}_{\text{fully antisymmetric}}
   \;\oplus\;
   \underbrace{V^{(2,1)}}_{\text{mixed symmetry}}.
$$

If we write 
$$
\,A^{(2)}_{\mathrm{sym}}:=P_{\mathrm{sym}}A^{(2)}\,,\qquad
\,A^{(2)}_{\mathrm{alt}}:=P_{\mathrm{alt}}A^{(2)}\,,\qquad
\,A^{(2)}_{\mathrm{mix}}:=P_{\mathrm{mix}}A^{(2)}\,.
$$
for the projection operators then the explicit formulas can for these operators can be given by 

$$
\begin{aligned}
\bigl(A^{(2)}_{\mathrm{sym}}\bigr)_{ijk}
 &=
 \tfrac16\!
 \Bigl(
   A_{ijk}+A_{jik}+A_{kij}+A_{ikj}+A_{jki}+A_{kji}
 \Bigr),\\[2pt]
\bigl(A^{(2)}_{\mathrm{alt}}\bigr)_{ijk}
 &=
 \tfrac16\!
 \Bigl(
   A_{ijk}-A_{jik}-A_{kij}-A_{ikj}+A_{jki}+A_{kji}
 \Bigr),\\[2pt]
\bigl(A^{(2)}_{\mathrm{mix}}\bigr)_{ijk}
 &=
   \frac{2}{3}A_{ijk} - \frac{1}{3}\left(A_{ikj} + A_{jki}\right) 
\end{aligned}
$$

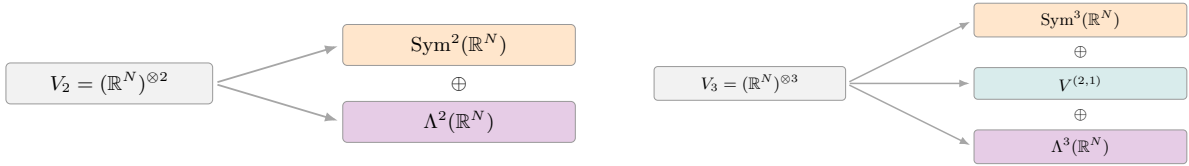
\begin{figure}[h]
\centering
\begin{minipage}{0.48\textwidth}
  \centering
  \scalebox{0.85}{%
  \begin{tikzpicture}[
    font=\small,
    box/.style={draw=gray!80, rounded corners=2pt, inner sep=4pt, minimum width=32mm, minimum height=6mm, fill=gray!10},
    comp/.style={draw=gray!80, rounded corners=2pt, inner sep=4pt, minimum width=36mm, minimum height=6mm},
    arr/.style={-{Latex[length=2mm]}, thick, shorten >=2pt, shorten <=2pt, draw=gray!70},
    symm_comp/.style={comp, fill=orange!20},
    anti_comp/.style={comp, fill=violet!20}
  ]
    \node[box] (V2) {$V_2=(\mathbb{R}^N)^{\otimes 2}$};

    \node[symm_comp, right=20mm of V2, yshift=6mm] (sym2) {$\mathrm{Sym}^2(\mathbb{R}^N)$};
    \node[anti_comp, right=20mm of V2, yshift=-6mm] (alt2) {$\Lambda^2(\mathbb{R}^N)$};

    \draw[arr] (V2.east) -- (sym2.west);
    \draw[arr] (V2.east) -- (alt2.west);

    \node at ($(sym2.south)!0.5!(alt2.north)$) {$\oplus$};
  \end{tikzpicture}
  } 
\end{minipage}%
\hfill
\begin{minipage}{0.48\textwidth}
  \centering
  \scalebox{0.7}{%
  \begin{tikzpicture}[
    font=\small,
    box/.style={draw=gray!80, rounded corners=2pt, inner sep=4pt, minimum width=36mm, minimum height=6mm, fill=gray!10},
    comp/.style={draw=gray!80, rounded corners=2pt, inner sep=4pt, minimum width=40mm, minimum height=6mm},
    arr/.style={-{Latex[length=2mm]}, thick, shorten >=2pt, shorten <=2pt, draw=gray!70},
    symm_comp/.style={comp, fill=orange!20},
    anti_comp/.style={comp, fill=violet!20},
    mix_comp/.style={comp, fill=teal!15}
  ]
    \node[box] (V3) {$V_3=(\mathbb{R}^N)^{\otimes 3}$};

    \node[symm_comp, right=24mm of V3, yshift=12mm] (sym3) {$\mathrm{Sym}^3(\mathbb{R}^N)$};
    \node[mix_comp, right=24mm of V3] (mix3) {$V^{(2,1)}$};
    \node[anti_comp, right=24mm of V3, yshift=-12mm] (alt3) {$\Lambda^3(\mathbb{R}^N)$};

    \draw[arr] (V3.east) -- (sym3.west);
    \draw[arr] (V3.east) -- (mix3.west);
    \draw[arr] (V3.east) -- (alt3.west);

    \node at ($(sym3.south)!0.5!(mix3.north)$) {$\oplus$};
    \node at ($(mix3.south)!0.5!(alt3.north)$) {$\oplus$};
  \end{tikzpicture}
  } 
\end{minipage}

\caption{Diagrams illustrating the orthogonal splitting of $V_2$ and $V_3$ into their isotypic components.}
\end{figure}

To quantify the symmetry properties of our tensors we employ the
\emph{Frobenius inner product}, which is orthogonal with respect to the splittings above.  For $V_2 = \mathbb{R}^{N \times N}$, the Frobenius inner product is given by 
$$
\langle A,B\rangle_F
   \;=\;
   \sum_{i,j=1}^{N}A_{ij}B_{ij}
   \;=\;
   \operatorname{Tr}(A^{\mathsf T}B).
$$
For $V_3=(\mathbb{R}^N)^{\otimes 3}$ we use
$$
\langle A,B\rangle_F
   \;=\;
   \sum_{i,j,k=1}^{N}A_{ijk}B_{ijk}.
$$
The induced norm is the \emph{Frobenius norm},
$\|A\|_F:=\sqrt{\langle A,A\rangle_F}$. A key fact is that the isotypic components are orthogonal with respect
to $\langle\,,\rangle_F$; e.g.\ for any
$S\in Sym^2(\mathbb{R}^N)$ and $L\in\Lambda^2(\mathbb{R}^N)$ we have
$\langle S,L\rangle_F=0$.  The same holds for the $k=3$ splitting. Orthogonality implies, via the Pythagorean theorem, that the squared
norm of a tensor equals the sum of the squared norms of its symmetry
components.  For
$A^{(1)}(t)=A^{(1)}_{\mathrm{sym}}(t)+A^{(1)}_{\mathrm{alt}}(t)$,
$$
\|A^{(1)}(t)\|_F^2
  \;=\;
  \|A^{(1)}_{\mathrm{sym}}(t)\|_F^2
  +
  \|A^{(1)}_{\mathrm{alt}}(t)\|_F^2.
$$
Likewise, for
$A^{(2)}(t)=A^{(2)}_{\mathrm{sym}}(t)
            +A^{(2)}_{\mathrm{alt}}(t)
            +A^{(2)}_{\mathrm{mix}}(t)$,
$$
\|A^{(2)}(t)\|_F^2
  \;=\;
  \|A^{(2)}_{\mathrm{sym}}(t)\|_F^2
  +
  \|A^{(2)}_{\mathrm{alt}}(t)\|_F^2
  +
  \|A^{(2)}_{\mathrm{mix}}(t)\|_F^2.
$$

Given the above orthogonal decompositions,  can now classify the long-term behaviour of our hypergraph by
tracking how the Frobenius norm is distributed among the symmetry
components as $t\to\infty$. \\

\begin{definition}[Asymptotic Regime]\label{def:asymptotic_regime}
Suppose we have an adaptive, triadic network dynamical system such as in Definition \ref{def:ATKHH}. We then say the hypergraph is either asymptotically symmetric, asymptotically antisymmetric, or mixed, depending on which of the following mutually exclusive criteria happen as $t \rightarrow \infty$.
\begin{itemize}
\item \textbf{Asymptotically Undirected/Symmetric.}
  The system is dominated by fully symmetric interactions; all
  antisymmetric and mixed parts vanish:
 $$
  \|A^{(1)}_{\mathrm{alt}}(t)\|_F\to0,
  \quad
  \|A^{(2)}_{\mathrm{alt}}(t)\|_F\to0,
  \quad
  \|A^{(2)}_{\mathrm{mix}}(t)\|_F\to0.
  $$

\item \textbf{Asymptotically Oriented/Antisymmetric.}
  The system is dominated by fully antisymmetric interactions; all
  symmetric and mixed parts vanish:
  $$
  \|A^{(1)}_{\mathrm{sym}}(t)\|_F\to0,
  \quad
  \|A^{(2)}_{\mathrm{sym}}(t)\|_F\to0,
  \quad
  \|A^{(2)}_{\mathrm{mix}}(t)\|_F\to0.
  $$

\item \textbf{Mixed regime.}
  At least one of the mixed-symmetry norms remains non-zero
  asymptotically.\\

\end{itemize}

\begin{center}
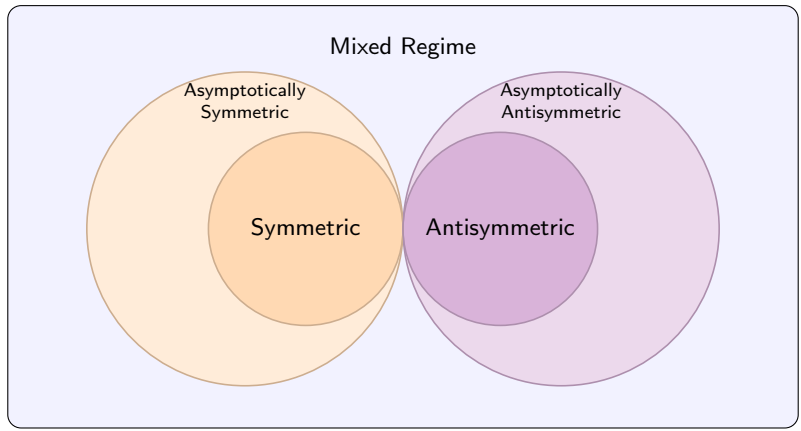
 
\begin{tikzpicture}[
    scale = 0.8,
  lab/.style={font=\small\sffamily},
  alabl/.style={font=\scriptsize\sffamily, align=center},
  outline/.style={line width=0.6pt}
]
  \colorlet{bg_color}{blue!5!white}
  \colorlet{symm_outer}{orange!15!white}
  \colorlet{symm_inner}{orange!30!white}
  \colorlet{asymm_outer}{violet!15!white}
  \colorlet{asymm_inner}{violet!30!white}
  
  \def\R{2.6} \def\r{1.6} \def\y{-0.2}
  
  \coordinate (CL) at (-\R, \y); \coordinate (CR) at (\R, \y);
  \coordinate (IL) at (-\r, \y); \coordinate (IR) at (\r, \y);

  \path[draw, rounded corners=5pt, fill=bg_color] (-6.5, -3.5) rectangle (6.5, 3.5);

  \path[fill=symm_outer] (CL) circle (\R); \path[fill=asymm_outer] (CR) circle (\R);
  \path[fill=symm_inner] (IL) circle (\r); \path[fill=asymm_inner] (IR) circle (\r);

  \draw[outline, color=symm_inner!80!black] (CL) circle (\R);
  \draw[outline, color=asymm_inner!80!black] (CR) circle (\R);
  \draw[outline, color=symm_inner!80!black] (IL) circle (\r);
  \draw[outline, color=asymm_inner!80!black] (IR) circle (\r);

  \node[lab] at (0, 2.8) {Mixed Regime};
  \node[lab, text=black] at (IL) {Symmetric};
  \node[lab, text=black] at (IR) {Antisymmetric};
  \node[alabl, text=black] at (-\R, \y + 2.1) {Asymptotically\\Symmetric};
  \node[alabl, text=black] at (\R, \y + 2.1) {Asymptotically\\Antisymmetric};
\end{tikzpicture}

\captionof{figure}{Schematic representation of the asymptotic regimes for adaptive triadic hypergraphs (Definition \ref{def:asymptotic_regime}). The outer rectangle denotes the mixed regime, in which symmetric and antisymmetric parts coexist with mixed components. The two large circles correspond to the asymptotically symmetric (left) and asymptotically antisymmetric (right) regimes, each containing an inner core of exactly symmetric or antisymmetric hypergraphs. All four circles meet at the zero hypergraph in the centre.}
\label{fig:regimes_schematic}
\end{center}
\end{definition}

\vspace{1em}

\begin{example}\label{example: symmetric case study}
    Suppose we have the adaptive, triadic Kuramoto model from Definition  \ref{def:ATKHH} given by 

    \begin{equation*}
    \dot{\theta}_i = \omega_i + \frac{1}{N} \sum_{j=1}^{N} A^{(1)}_{ij} f_1(\theta_i, \theta_j)
    + \frac{1}{N^2} \sum_{k=1}^N \sum_{j=1}^{N} A^{(2)}_{ijk} \,  f_2(\theta_i, \theta_j, \theta_k),
\end{equation*}
with the following dynamics on the adjacency matrix and tensor,
\begin{equation*}
    \dot{A}^{(1)}_{ij} = -\delta_1 \left(A^{(1)}_{ij} + \cos(\theta_i - \theta_j ) \right).
\end{equation*}
\begin{equation*}
    \dot{A}^{(2)}_{ijk} = -\delta_2 \left(A^{(2)}_{ijk} + \cos(\theta_i + \theta_j + \theta_k ) \right),
\end{equation*}

where $\delta_1, \delta_2$ are some positive constants. We show this system converges to the asymptotic undirected regime. First let $A^{(1)} = \left( A^{(1)}_{ij} \right)_{1 \leq i,j \leq N }$ and let $M = \left( \cos(\theta_i - \theta_j) \right)_{1 \leq i ,j \leq N } \in \mathbb{R}^{N \times N} .$ Then we can rewrite the binary interaction ODEs into one matrix form, given by   $ \dot A^{(1)} = -\delta_1 (A^{(1)} + M).$ Since cosine is an even function, $M$ is a symmetric matrix, hence $P_{sym}(M) = M.$ Applying the symmetric and antisymmetric projection operators to the binary matrix evolution equation therefore yields

 $$\dot A^{(1)}_{sym} = - \delta_1( A^{(1)}_{sym} + M ),$$
 $$\dot A^{(1)}_{anti} = -\delta_1 A^{(1)}_{anti}.$$

 We  see that $  A^{(1)}_{anti} =  A^{(1)}_{anti}(0) e^{-\delta_1 t}$ and so $\lim_{t \rightarrow \infty}  A^{(1)}_{anti} = 0.$ In fact, the symmetric component can be solved, yielding
\begin{equation}\label{equation: symmetric binary}
A^{(1)}_{sym}(t)
   \;=\;
   e^{-\delta_1 t}A^{(1)}_{sym}(0)
   \;-\;
   \delta_1
   \int_{0}^{t}
      e^{-\delta_1 (t-s)}\,M(s)\,ds,
\end{equation}

For the triadic interactions, introduce the tensors  
$
A^{(2)}=\bigl(A^{(2)}_{ijk}\bigr)_{1\le i,j,k\le N},
T_{ijk}:=\cos\bigl(\theta_i+\theta_j+\theta_k\bigr).
$
The evolution law  for the triadic terms can consequently be written as  
$
\dot A^{(2)}=-\delta_2\bigl(A^{(2)}+T\bigr).
$ Because the argument $\theta_i+\theta_j+\theta_k$ is invariant under every permutation of $(i,j,k)$, i.e.   
$
T_{ijk}=T_{\sigma(i)\sigma(j)\sigma(k)}$  for all $\sigma\in S_3,
$
we have that  $T\in\text{Sym}^3(\mathbb{R}^N)$.  Consequently  
$
P_{sym}(T)=T,  P_{alt}(T)=0, P_{mix}(T)=0.
$

Applying the Young-symmetriser projections  $(P_{sym},P_{alt},P_{mix})$  to the tensor ODE therefore yields  
$$
\begin{aligned}
\dot A^{(2)}_{sym} &= -\delta_2\bigl(A^{(2)}_{sym}+T\bigr),\\
\dot A^{(2)}_{alt} &= -\delta_2\,A^{(2)}_{alt},\\
\dot A^{(2)}_{mix} &= -\delta_2\,A^{(2)}_{mix}.
\end{aligned}
$$

The antisymmetric and mixed parts have solutions which decay to zero given by 
$ A^{(2)}_{alt}(t)=A^{(2)}_{alt}(0)\,e^{-\delta_2 t}, 
A^{(2)}_{mix}(t)=A^{(2)}_{mix}(0)\,e^{-\delta_2 t}.
$
We can solve for the symmetric part, yielding
\begin{equation*}\label{equation: symmetric triadic equation}
A^{(2)}_{sym}(t)
   = e^{-\delta_2 t}A^{(2)}_{sym}(0)
     -\delta_2\int_{0}^{t} e^{-\delta_2(t-s)}\,T(s)\,ds,
\end{equation*}
and just like in the binary case, this does not necessarily decay to zero. Just as on the level of binary interactions, the symmetric term Frobenius strength has the following estimate, 

$$0 \leq \| A^{(2)}_{sym}(t)\|
   \leq  e^{-\delta_2 t}\|A^{(2)}_{sym}(0)\| + 
     \left( 1 - e^{-\delta_2 t} \right)\sup_{s \geq 0}\| T(s)\|.$$

From this we see that our hypergraph converges to the undirected (symmetric) regime, in the sense of Definition \ref{def:asymptotic_regime}. Furthermore, suppose we choose specific functions for $f_1$ and $f_2$ given by 

$$f_1(\theta_i, \theta_j) = \sin(\theta_i - \theta_j),$$
$$f_2(\theta_i, \theta_j, \theta_k) = \sin(2\theta_i - \theta_j -\theta_k).$$

We numerically simulate the Frobenius norms and observe that everything but the symmetric parts converge to zero, confirming our analytic findings, see figures \ref{fig:binary_energy} and \ref{fig:triadic_energy}. The simulations were run over $t \in [0,50]$ with $500$ evaluation points. We used $N=5$ oscillators with natural frequencies $\omega_i \sim \mathcal{N}(0,1)$, initial phases drawn uniformly from $[0,2\pi)$, and random initial weights. Both $\delta_1$, and $ \delta_2$ were set to $ 0.1$. The full implementation is available in the code repository.\\

\begin{figure}[t]
  \centering

  \begin{subfigure}[b]{0.47\textwidth}
    \begin{overpic}[width=\linewidth]{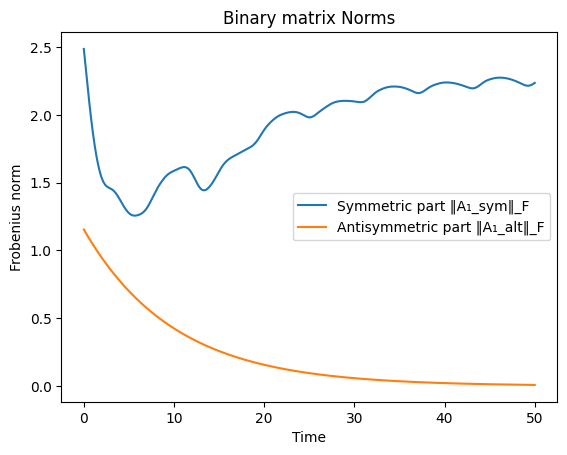}
    \end{overpic}
    \caption{Binary layer. Time series of the Frobenius norms
      $\norm{A^{(1)}_{\mathrm{sym}}}_{F}$ and $\norm{A^{(1)}_{\mathrm{alt}}}_{F}$ obtained by projecting
      $A^{(1)}$ onto the symmetric and antisymmetric subspaces of $V_{2}=M_{N}(\mathbb{R})$.
      The antisymmetric energy decays like $e^{-\delta_{1} t}$ with $\delta_{1}=0.1$, while the symmetric
      part saturates at the norm of the driver $M=\cos(\theta_{i}-\theta_{j})$, consistent with
      \eqref{equation: symmetric binary}.}
    \label{fig:binary_energy}
  \end{subfigure}
  \hfill
  \begin{subfigure}[b]{0.47\textwidth}
    \begin{overpic}[width=\linewidth]{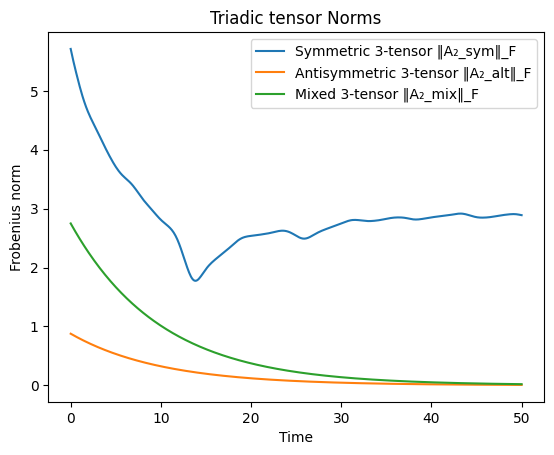}
    \end{overpic}
    \caption{Triadic layer. Frobenius norms of the symmetric, antisymmetric, and mixed components of
      $A^{(2)}$. Both $\norm{A^{(2)}_{\mathrm{alt}}}_{F}$ and $\norm{A^{(2)}_{\mathrm{mix}}}_{F}$ decay like
      $e^{-\delta_{2} t}$ with $\delta_{2}=0.1$, whereas the symmetric energy converges to the norm of the driver
      $T=\cos(\theta_{i}+\theta_{j}+\theta_{k})$, in agreement with
      \eqref{equation: symmetric triadic equation}.}
    \label{fig:triadic_energy}
  \end{subfigure}

  \caption{Symmetry–component norms in the binary and triadic layers.} 
  \label{fig:energy_plots}
\end{figure}

\end{example}

\begin{example}[Antisymmetric case study]\label{example:antisymmetric}
We construct a system that flows into the oriented
(antisymmetric) regime, in the sense of Definition \ref{def:asymptotic_regime}. First we recall the Levi-Civita symbol.\\

\begin{definition}[Generalised Levi–Civita symbol]\label{def:levi-civita}
For any three indices $i,j,k\in\{1,\dots ,N\}$ (with $N\ge 3$) define
$$
\varepsilon_{ijk} \;=\;
\begin{cases}
\; 0 &
\text{if } i=j \text{ or } j=k \text{ or } k=i,\\[6pt]
\; +1 &
\text{if } (i,j,k) \text{ is an \emph{even} permutation of its
           ascending reordering }(i',j',k'),\\[6pt]
\; -1 &
\text{if } (i,j,k) \text{ is an \emph{odd} permutation of }(i',j',k').
\end{cases}
$$
Here $(i',j',k')$ is the unique ordering of $\{i,j,k\}$ with
$i'<j'<k'$, and the parity is taken with respect to
that ascending triple.  Thus $\varepsilon_{ijk}$ is totally
antisymmetric, i.e. 
$\varepsilon_{\sigma(i)\sigma(j)\sigma(k)}
 =\operatorname{sgn}(\sigma)\,\varepsilon_{ijk}$ for every
$\sigma\in S_3$.
\end{definition}

\color{black}
We allow our model to have the following dynamics on the nodes, edges, and hyperedges.
 \begin{equation*}
    \dot{\theta}_i = \omega_i + \frac{1}{N} \sum_{j=1}^{N} A^{(1)}_{ij} f_1(\theta_i, \theta_j)
    + \frac{1}{N^2} \sum_{k=1}^N \sum_{j=1}^{N} A^{(2)}_{ijk} \,  f_2(\theta_i, \theta_j, \theta_k),
\end{equation*}
with the following dynamics on the heterogeneous hypergraph, 
\begin{equation*}
    \dot{A}^{(1)}_{ij} = -\delta_1 \left(A^{(1)}_{ij} + \sin(\theta_i - \theta_j ) \right), \qquad 
    \dot{A}^{(2)}_{ijk} = -\delta_2 \left(A^{(2)}_{ijk} + \varepsilon_{ijk}\sin(\theta_i + \theta_j + \theta_k ) \right).
\end{equation*}

On the level of binary interactions, first define the matrix $Q \in \mathbb{R}^{N \times N}$ as 
$
\left(Q_{ij}\right) \;=\; \left(\sin(\theta_i-\theta_j) \right)_{1 \leq i,j \leq N}\;=\;\left(-Q_{ji}\right).
$
Because $P_{\text{alt}}(N)=N$ and $P_{\text{sym}}(N)=0$, projection
yields the decoupled equations
$$
\dot A^{(1)}_{\text{alt}} = -\delta_1\bigl(A^{(1)}_{\text{alt}}-N\bigr),
\qquad
\dot A^{(1)}_{\text{sym}} = -\delta_1\,A^{(1)}_{\text{sym}}.
$$

The symmetric bit decays to zero, and similarly to the previous example we have that 
\begin{equation}\label{equation: antisymmetric binary}
A^{(1)}_{alt}(t)
   \;=\;
   e^{-\delta_1 t}A^{(1)}_{alt}(0)
   \;-\;
   \delta_1
   \int_{0}^{t}
      e^{-\delta_1 (t-s)}\,Q(s)\,ds.
\end{equation}

For the triadic interactions, introduce the totally antisymmetric driver
$$
H_{ijk} \;=\;
\varepsilon_{ijk}\,\sin\bigl(\theta_i+\theta_j+\theta_k\bigr),
$$
and let $H \in \left( \mathbb{R}^N \right)^{\otimes 3}$ be defined as $H = \left( H_{ijk} \right)_{1 \leq i,j,k \leq N}.$ Because $P_{\text{alt}}(H)=H$ while
$P_{\text{sym}}(H)=P_{\text{mix}}(H)=0$, projecting onto the triadic system of ODEs gives
$$
\dot A^{(2)}_{\text{alt}}
   = -\delta_2\bigl(A^{(2)}_{\text{alt}}-H\bigr),
\qquad
\dot A^{(2)}_{\text{sym}}
   = -\delta_2\,A^{(2)}_{\text{sym}},
\qquad
\dot A^{(2)}_{\text{mix}}
   = -\delta_2\,A^{(2)}_{\text{mix}}.
$$
It follows the symmetric and mixed terms decay to zero, and the mixed term does not necessarily.

Since both the binary and triadic layers lose every component outside the fully antisymmetric subspaces in the limit,  the directed hypergraph converges to the
antisymmetric regime. Similar to the symmetric case, we simulate these dynamics over $t \in [0,50]$ with 500 evaluation points. We used $N=5$ oscillators with frequencies 
$\omega_{i}$ equally spaced in $[-1,1]$, random initial weights, initial phases sampled 
uniformly from $[0,1]$, and set $\delta_{1}=\delta_{2}=0.1$, see \ref{fig:bin_antisym} and \ref{fig:tri_antisym}. The full implementation is available in the code repository.

\begin{figure}[htbp]
  \centering

  \begin{subfigure}[b]{0.48\textwidth}
    \begin{overpic}[width=\linewidth]{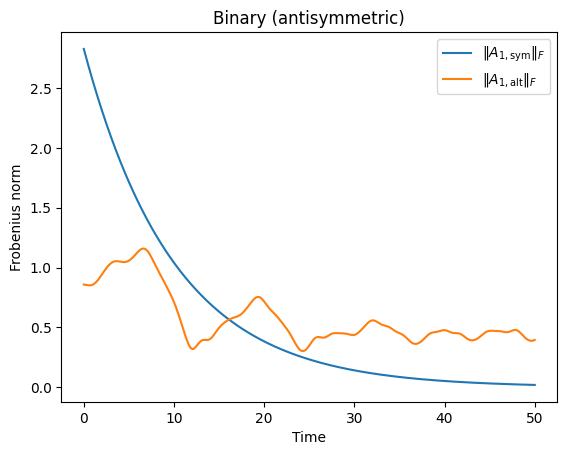}
    \end{overpic}
    \caption{Binary layer. Frobenius norms of the symmetric and
      antisymmetric components of the evolving edge matrix
      $A^{(1)}(t)$ under the antisymmetric driver
      $N_{ij}=\sin(\theta_{i}-\theta_{j})$. As shown analytically,
      $\norm{A^{(1)}_{\mathrm{alt}}}_{F}$ does not decay to zero,
      unlike $\norm{A^{(1)}_{\mathrm{sym}}}_{F}$, which decays like
      $e^{-\delta_{1} t}$ for $\delta_{1}=0.1$.}
    \label{fig:bin_antisym}
  \end{subfigure}
  \hfill
  \begin{subfigure}[b]{0.48\textwidth}
    \begin{overpic}[width=\linewidth]{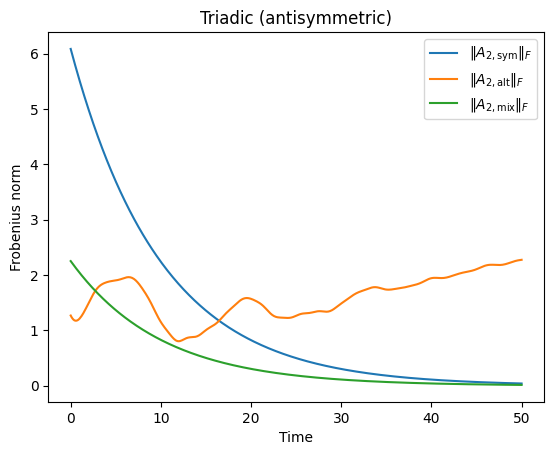}
    \end{overpic}
    \caption{Triadic layer. Frobenius norms of the fully symmetric,
      fully antisymmetric, and mixed parts of $A^{(2)}(t)$ driven by
      $H_{ijk}=\varepsilon_{ijk}\sin(\theta_{i}+\theta_{j}+\theta_{k})$.
      Both $\norm{A^{(2)}_{\mathrm{sym}}}_{F}$ and
      $\norm{A^{(2)}_{\mathrm{mix}}}_{F}$ vanish exponentially
      (like $e^{-\delta_{2} t}$ with $\delta_{2}=0.1$).}
    \label{fig:tri_antisym}
  \end{subfigure}

    \caption{Convergence to the antisymmetric regime. (a) Binary layer. (b) Triadic layer.}
    \label{fig:antisym_combined} 
\end{figure}

\end{example}

\begin{remark}
The two examples above show that the weight tensors encoding our hypergraph asymptotically collapse onto distinct $S_2$ and $S_3$ isotypic components. These regimes  have consequences in terms of which types of simplicial complexes we can and cannot study the emergence of, on our heterogeneous hypergraphs. Such a classification is summarised as follows.

\begin{itemize}
  \item \textbf{Symmetric regime:} When the hypergraph  enters the symmetric regime, the antisymmetric and mixed terms become negligible. Hence after a sufficient interval of time, the matrix and tensor encoding our hypergraph is fully symmetric (up to some small $\varepsilon.$)
        Thus, each interaction depends solely on the unordered vertex set.  In this case, we can study the emergence of standard, unoriented simplicial complexes. We do this using an $\delta$ boundary test in the next section. Once a simplicial complex has been established and is stable, the classical tools such as simplicial and persistent homology are thereafter valid.
        \item \textbf{Oriented regime:} When the hypergraph settles into the antisymmetric regime, the symmetric and mixed contributions fall below the tolerance~$\varepsilon$.  All surviving interactions are alternating, so they are encoded by ordered vertex tuples, up to an overall sign.  From here, we employ a threshold test, now with orientation, to certify that every active oriented $p$-simplex carries all of its $p-1$ faces. Once this is done, we can employ (oriented) simplicial homology. 

\item \textbf{Mixed regime:} If mixed-symmetry terms persist, or if symmetric and antisymmetric parts coexist above $\varepsilon$, no ordering collapses the data into a downward-closed family.  Hence, neither unoriented nor oriented simplicial complexes emerge. The minimal language that still captures the higher-order structure is that of simplicial sets or $\Delta$-sets, whose construction and invariants we postpone to a later section.

\end{itemize}
\end{remark}

%% file: Symmetric_regime.tex
\section{Simplicial Complex Emergence in the Symmetric Regimes.}
In this section, we formulate results that guarantee invariance and formation of simplicial complexes for hypergraphs that converge to the symmetric and antisymmetric regimes.

\subsection{Retention and Emergence of Simplicial Structure}
Suppose we have an adaptive higher order model such as in Definition \ref{def:ATKHH}  with a pairwise weighted adjacency matrix $A^{(1)}(t)$ and weighted rank-$3$ tensor $A^{(2)}(t)$ encoding the higher order interactions. Suppose that it converges to the symmetric regime, in the sense of Definition \ref{def:asymptotic_regime}. Thus, for all $\varepsilon > 0$ there exists a time $t_0$ such that for all $t \geq t_0$ we have that all the antisymmetric and mixed parts decay to zero, and both $|A^{(1)}_{ij}(t) - A^{(1)}_{ji}(t)| < \varepsilon$ and $|A^{(2)}_{ijk}(t) - A^{(2)}_{\sigma(i)\sigma(j)\sigma(k)}(t)| < \varepsilon$ for all $\sigma \in S_3$. In such a case, we create equivalence classes and view $A^{(1)}_{ij}$ and $A^{(1)}_{ji}$ as the same object. Similarly, $A^{(2)}_{ijk}$ will be the representative for all other permutations of its indices, which is a well-defined notion since they also differ only by a small $\varepsilon$, which we have total control over.  

We are interested in the formation of a simplicial complex structure on our hypergraph.  We first note that the nodal variables $x_i(t)$ describe the internal state (e.g. phase or activity) of each agent rather than its structural presence. Consequently, the set of $0$-simplicies is taken to be fixed, with adaptivity acting only on higher-order simplicies.  Hence, when trying to see whether or not we have a simplicial complex, we need not concern ourselves with the behaviour of the $0$-simplices, i.e. vertices. Therefore, to ensure a simplicial complex structure, we require that all subfaces of a degree $2$ face are non-zero; that is, we require that if $A^{(2)}_{ijk} \neq 0$, for some $i,j,k$ then $A^{(1)}_{\ell p} \neq 0$ for all $\{\ell, p\} \subseteq \{i , j , k\}.$ Thus, we want to examine the set that violates the simplicial complex structure, which is given by 
\begin{equation}\label{equation: first manifold}
    M_{ijk} := \{ (A^{(2)}_{ijk}, A^{(1)}_{ij}, A^{(1)}_{ik}, A^{(1)}_{jk} ) \in \mathbb{R}^4 \mid A^{(1)}_{ij} A^{(1)}_{ik} A^{(1)}_{jk} = 0 \text{ and } A^{(2)}_{ijk} \neq 0 \}.
\end{equation}

In practice, the all-or-nothing condition
$$
   A^{(2)}_{ijk} \neq 0
   \quad\text{and}\quad
   A^{(1)}_{ij} \cdot \,A^{(1)}_{ik} \cdot \,A^{(1)}_{jk} = 0
$$
is too rigid for deciding whether a genuine $2$-simplex is present on the
nodes $i,j,k.$  
Consider the extreme case in which $A^{(2)}_{ijk} \gg 1$ but, for instance,
$A^{(1)}_{ij} = 10^{-5}$.  Under any realistic notion of presence, that
triangle should still be regarded as absent. To encode this
intuition we introduce a tolerance parameter $\delta > 0$: interactions
below $\delta$ are treated as effectively zero, while interactions
exceeding~$\delta$ are treated as present. We therefore make the following definition. \\

\begin{definition}
We say the triadic connection strength $A^{(2)}_{ijk}$ respects the downward closure property with parameter $\delta > 0,$ if and only if 
\begin{equation}\label{def: 2 simplex with parameter delta}
    |A^{(2)}_{ijk}| \geq \delta \implies |A^{(1)}_{ij}|  \geq \delta, |A^{(1)}_{ik}| \geq \delta, |A^{(1)}_{jk}| \geq \delta.
\end{equation}
    
    We say a $2-$simplex is present on the nodes $i,j$ and $k$ with parameter $\delta$, if and only if 
    $|A^{(2)}_{ijk}| \geq \delta $
    and $A^{(2)}_{ijk}$ respects the downward closure property.
\end{definition}

We say a $1-$simplex is present on the nodes $i$ and $j$, with parameter $\delta$, if $|A^{(1)}_{ij}| \geq \delta. $ Since the vertices are assumed vacuously present (and of strength greater than $\delta),$ we assume the $1-$simplices respect the downward closure property automatically.\\

\begin{definition}
Given a directed hypergraph $\mathcal{H}(t) = (\mathbf{x}(t), A^{(1)}(t), A^{(2)}(t))$ that converges to the symmetric regime as in Definition \ref{def:asymptotic_regime}, we say $\mathcal{H}(t)$ is a simplicial complex with parameter $\delta > 0$, if for all $i,j,k \in \{ 1 , \dots , N\}$, the downward closure property \ref{def: 2 simplex with parameter delta} is not violated. 
\end{definition}

This softened rule provides a more faithful measure of when simplicial dynamics should, or
should not, be deemed observable. We now want to study when a simplicial complex is present on our hypergraph, and under what conditions is it stable. As such, we study the following modified subset of $\mathbb{R}^4.$\\

\begin{definition}
    Let $i,j,k \in \{ 1, \dots , N\}.$ Then the set of points where we do not have a simplicial complex with parameter $\delta$ on these nodes, denoted by  
    $M_{ijk}^{2,\delta} \subseteq \mathbb{R}^4$, is given by 
    \begin{equation}\label{equation: definition of M_epsilon}
         M_{ijk}^{2,\delta} = \left\{ \bigl(A^{(2)}_{ijk}, A^{(1)}_{ij}, A^{(1)}_{ik}, A^{(1)}_{jk},\bigr) \in \mathbb{R}^4 \mid |A^{(2)}_{ijk}| \geq \delta \text{ and } \left(|A^{(1)}_{ij}| < \delta \text{ or } |A^{(1)}_{ik}| < \delta \text{ or } |A^{(1)}_{jk}| < \delta\right)\right\}.
    \end{equation}
\end{definition}

\begin{figure}[h]
\centering
\begin{tikzpicture}[
  scale=0.9,
  vtx/.style={circle, fill=black, inner sep=1.7pt},
  face/.style={fill=gray!20}, 
  solidE/.style={draw=blue!70, line width=1.2pt},
  weakE/.style={draw=blue!70, line width=1.2pt, dashed},
  lab/.style={font=\small\color{blue!70}}
]
  \coordinate (i) at (0,0);
  \coordinate (j) at (6,0);
  \coordinate (k) at (3,4.2);

  \path[face] (i)--(j)--(k)--cycle;

  \draw[solidE] (i) -- (j)
    node[lab, midway, below] {$|A^{(1)}_{ij}|\geq \delta$};
  \draw[solidE] (i) -- (k)
    node[lab, midway, left] {$|A^{(1)}_{ik}|\geq \delta$};
  \draw[weakE] (j) -- (k)
    node[lab, midway, right] {$|A^{(1)}_{jk}|<\delta$};

  \node[vtx] at (i) {}; \node[below left=2pt of i] {$i$};
  \node[vtx] at (j) {}; \node[below right=2pt of j] {$j$};
  \node[vtx] at (k) {}; \node[above=2pt of k] {$k$};

  \node[lab] at (barycentric cs:i=1,j=1,k=1) {$|A^{(2)}_{ijk}|\geq \delta$};
\end{tikzpicture}

\caption{Violation of downward closure: the triad $(i,j,k)$ is above threshold ($|A^{(2)}_{ijk}|\geq \delta$), edges $(i,j)$ and $(i,k)$ are above threshold, but $(j,k)$ falls below threshold ($|A^{(1)}_{jk}|<\delta$), so this is not a simplicial complex.}
\end{figure}
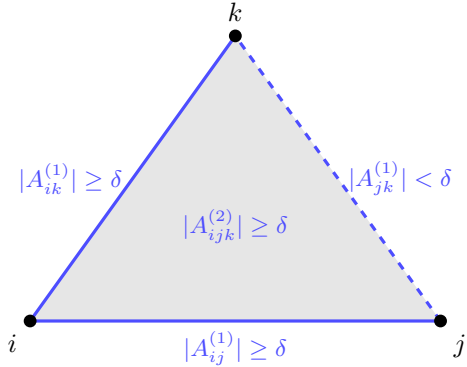

 In order for a simplicial complex to emerge, we seek to find conditions on our dynamics so that the above set, equation \ref{equation: definition of M_epsilon}, is negatively time-invariant.\\

 \begin{definition}[Invariance]
Let $\dot{x} = g(x)$ be a dynamical system on $\mathbb{R}^n$ with flow $\phi_t$. 
A set $S \subset \mathbb{R}^n$ is called \emph{invariant} for $\phi_t$ if 
$$
x_0 \in S \;\;\Longrightarrow\;\; \phi_t(x_0) \in S \quad \text{for all } t \in \mathbb{R}.
$$
It is called \emph{positively invariant} if this implication holds for all $t \geq 0$, and 
\emph{negatively invariant} if it holds for all $t \leq 0$. 
\end{definition}

 To this end, it would at first appear tempting to try and find a smooth function $f \in C^\infty(\mathbb{R}^4)$ whose vanishing set, given by $f^{-1}(0) := \{x \in \mathbb{R}^4 \mid f(x) = 0\}$, equals the boundary of $M_{ijk}^{2,\delta}$, i.e. $\partial M_{ijk}^{2,\delta}$. One could then calculate the outward-pointing normal vector from $M_{ijk}^{2,\delta}$ by simply calculating the gradient of $f$, i.e. $\nabla f$. If the flow on the space is given by $\dot{x} = g(x)$, one could then take the Lie derivative with the flow, given by 
$$
    \nabla f \cdot g,
$$ 
to determine if particles are exiting or entering the region, depending on the sign of the Lie derivative. This approach is global in nature as one such theoretical function $f$ could provide all the information required to analyse the boundary. However, the boundary $\partial M_{ijk}^{2,\delta}$ is defined piecewise by different inequalities (for example, parts where $|A_{ij}^{(1)}| < \delta$ and parts where $|A_{ijk}^{(2)}| \geq \delta$). As a result, these pieces do not glue together to form a single smooth manifold. Consequently, the regular value theorem\footnote{Recall the \textbf{Regular Value Theorem}: If $f \colon \mathbb{R}^n \to \mathbb{R}^m$ is smooth and $c \in \mathbb{R}$ is a regular value (meaning that for every $x \in f^{-1}(c)$ the derivative $Df(x)$ has full rank), then $f^{-1}(c)$ is a smooth submanifold of $\mathbb{R}^n$ with dimension $n - m$, provided it is non-empty.} implies that the gradient of such a hypothetical function $f$ would vanish everywhere on $\partial M_{ijk}^{2,\delta}$, i.e. $\nabla f|_{\partial M_{ijk}^{2,\delta}} = 0$. This approach therefore provides no information on the Lie derivative, and so such a global approach fails. 

As such, we attempt a local approach, and first calculate the boundary $\partial M_{ijk}^{2,\delta}$ explicitly, before deriving the necessary conditions for outward flow from the boundary. 

We first recall that 
$$
    \partial M_{ijk}^{2,\delta} = \overline{M_{ijk}^{2,\delta}} \setminus \text{int}\bigl(M_{ijk}^{2,\delta}\bigr).
$$

Next, we note that 
$$
    \overline{M_{ijk}^{2,\delta}} = \{|A^{(2)}_{ijk}| \geq \delta,\, |A^{(1)}_{ij}| \leq \delta\} 
    \cup \{|A^{(2)}_{ijk}| \geq \delta,\, |A^{(1)}_{ik}| \leq \delta\} 
    \cup \{|A^{(2)}_{ijk}| \geq \delta,\, |A^{(1)}_{jk}| \leq \delta\},
$$
and that 
$$
    \text{int}\bigl(M_{ijk}^{2,\delta}\bigr) = \left\{|A^{(2)}_{ijk}| > \delta \text{ and } \left(|A^{(1)}_{ij}| < \delta \text{ or } |A^{(1)}_{ik}| < \delta \text{ or } |A^{(1)}_{jk}| < \delta\right)\right\}.
$$

One can then see that 
\begin{align*}
\{|A^{(2)}_{ijk}| \geq \delta,\, |A^{(1)}_{ij}| \leq \delta\} \setminus \text{int}\bigl(M_{ijk}^{2,\delta}\bigr)
    &= \{|A^{(2)}_{ijk}| = \delta,\, |A^{(1)}_{ij}| \leq \delta\} \\
    &\quad\cup \{|A^{(2)}_{ijk}| \geq \delta,\, |A^{(1)}_{ij}| = \delta,\, |A^{(1)}_{ik}| \geq \delta,\, |A^{(1)}_{jk}| \geq \delta\}.
\end{align*}
By symmetry, similar expressions hold for the other two sets making up the closure. Therefore, the full boundary is:

\begin{align}\label{equation: boundary of bad set}
\partial M_{ijk}^{2,\delta} 
    &= 
    \underbrace{\bigl\{|A^{(2)}_{ijk}| = \delta,\, \min\bigl(|A^{(1)}_{ij}|,|A^{(1)}_{ik}|,|A^{(1)}_{jk}|\bigr)\leq \delta \bigr\}}_{ =: X_1} \\
    &\quad\cup 
    \underbrace{\{|A^{(2)}_{ijk}| \geq \delta,\, |A^{(1)}_{ij}| = \delta,\, |A^{(1)}_{ik}| \geq \delta,\, |A^{(1)}_{jk}| \geq \delta\}}_{=:X_2} \\
    &\quad\cup 
    \underbrace{\{|A^{(2)}_{ijk}| \geq \delta,\, |A^{(1)}_{ik}| = \delta,\, |A^{(1)}_{ij}| \geq \delta,\, |A^{(1)}_{jk}| \geq \delta\}}_{=:X_3} \\
    &\quad\cup 
    \underbrace{\{|A^{(2)}_{ijk}| \geq \delta,\, |A^{(1)}_{jk}| = \delta,\, |A^{(1)}_{ij}| \geq \delta,\, |A^{(1)}_{ik}| \geq \delta\}}_{=:X_4}.
\end{align}

Thus, we can write $\partial M_{ijk}^{2,\delta} = \cup_{i = 1,2,3,4} X_i$ where $X_i$ are mutually disjoint almost everywhere, with the $X_i$ written as in \ref{equation: boundary of bad set}. From this, we get that the outward-pointing normal vector to our boundary is  

$$n : \partial M_{ijk}^{2,\delta} \rightarrow \mathbb{R}^4, $$
$$n\bigl(A^{(2)}_{ijk},A^{(1)}_{ij},A^{(1)}_{ik},A^{(1)}_{jk}\bigr)  = \begin{cases} 
    (-\text{sgn}(A^{(2)}_{ijk}), 0, 0, 0) & \text{ on } X_1 \\ 
    \bigl(0, \text{sgn}(A^{(1)}_{ij}), 0, 0\bigr) & \text{ on } X_2\\
    \bigl(0, 0, \text{sgn}(A^{(1)}_{ik}), 0\bigr) & \text{ on } X_3\\
     \bigl(0, 0, 0, \text{sgn}(A^{(1)}_{jk})\bigr) & \text{ on } X_4
\end{cases}$$

\begin{remark}
If we begin with a simplicial complex and ensure that the above outward normal vector has a non-negative scalar product with the vector field, then we will ensure our hypergraph cannot lose its simplicial complex structure. 
One technicality lies on the points lying on at least two of the $X_i$. The collection of all such points lies on intersections of affine hyperplanes and hence forms a set of measure zero. Moreover, they do not admit a unique outward normal vector, and for this reason, we call them \emph{singular}.  We can still verify outward flow from the boundary at these singular points as follows. 
To be more precise, suppose $x^*$ is one such singular and point and let $I(x^*)=\{\,i\in\{1,2,3,4\} \mid  x^*\in X_i\}.$ The outward normal cone  at $x^*$ is then the non-negative hull 
$$N_{\partial M}(x^*) = 
\left\{\sum_{i\in I(x^{\ast})}\lambda_i\,n_i(x^{\ast}) \ \Big|\ \lambda_i \in \mathbb{R}_{\geq 0} \right\}.$$

This allows us to formulate a precise definition of outward pointing with respect to the set $M^{2,\delta}_{ijk}.$\\

\end{remark}

\begin{definition}[Outward–pointing]\label{def: outward pointing symmetric}
Suppose $\mathcal{H}(t)$ is an adaptive, triadic network dynamical system encoded by a matrix $A^{(1)}$ and rank-$3$ tensor $A^{(2)}$, and suppose it converges to the symmetric regime as in Definition \ref{def:asymptotic_regime}. Let $F^{(ijk)}
:= \bigl(\dot A^{(2)}_{ijk},\, \dot A^{(1)}_{ij},\, \dot A^{(1)}_{ik},\, \dot A^{(1)}_{jk}\bigr)$ We then say the system is \emph{outward–pointing at $(i,j,k)$ with threshold $\delta$} if
$$
n\cdot F^{(ijk)}(x)\ \geq\ 0\qquad \text{for every }x\in \partial M^{2,\delta}_{ijk}\text{ and every }n\in N_{\partial M}(x).
$$
Equivalently, it suffices to verify the sign–derivative inequalities
$$
\begin{aligned}
&\text{on }X_1:\quad \text{sgn}(A^{(2)}_{ijk})\,\dot A^{(2)}_{ijk}\leq 0,\\
&\text{on }X_2:\quad \text{sgn}(A^{(1)}_{ij})\,\dot A^{(1)}_{ij}\geq 0,\qquad
\text{on }X_3:\ \text{sgn}(A^{(1)}_{ik})\,\dot A^{(1)}_{ik}\geq 0,\qquad
\text{on }X_4:\ \text{sgn}(A^{(1)}_{jk})\,\dot A^{(1)}_{jk}\geq 0.
\end{aligned}
$$
\end{definition}

Importantly, our Definition  \ref{def: outward pointing symmetric} of outward pointing ensures that the singular points on $\partial M^{2,\delta}_{ijk}$ have outward pointing normal cones.\\

\begin{remark}
In order for the notion of outward--pointing to be well defined, we require that the 
binary and triadic dynamics admit continuous trajectories. 
Concretely, we assume throughout that the right--hand sides 
$$
\dot A^{(1)}_{ij} = g^{(1)}_{ij}(t,A^{(1)},A^{(2)},x), \qquad
\dot A^{(2)}_{ijk} = g^{(2)}_{ijk}(t,A^{(1)},A^{(2)},x)
$$
are at least $C^1$. However, as we will see later, many examples of such dynamics will involve terms with an absolute value appearing; hence, we only require $C^1$ away from the vanishing hyperplanes. More precisely, let 

$$\mathcal{H}_0 := 
    \bigcup_{1\leq i<j\leq N}\{A^{(1)}_{ij}=0\}
    \;\cup\;
    \bigcup_{1\leq i<j<k\leq N}\{A^{(2)}_{ijk}=0\};$$
Then we require the dynamics are $C^1$ on $\mathbb{R}^d \setminus \mathcal{H}_0$ where $d = N + N^2 + N^3.$\\

\end{remark}

\begin{definition}[Simplicial region at threshold $\delta$]
Let $K_{ijk}:=\mathbb{R}^4\setminus M^{2,\delta}_{ijk}$. Define the global simplicial region for our system as 
$$
\Omega_\delta:=\prod_{i<j<k} K_{ijk}\ \subseteq  \mathbb{R}^{4\binom{N}{3}}.
$$
\end{definition}

 The ambient space is $\mathbb{R}^{4  \binom{N}{3}}$ since each $K_{ijk}$ contains a $4-$tuple, and there are $\binom{N}{3}$ such sets $K_{ijk}$; this follows because in the symmetric regime each $A^{(2)}_{ijk}$ is invariant under any $S_3$ permutation of the indices.  We can now state our main result of this section. \\

\begin{theorem}[Retention of simplicial structure]\label{thm:retention}
Fix a simplicial tolerance parameter $\delta > 0$. 
Assume $\mathcal{H}(t) = (\textbf{x},A^{(1)},A^{(2)}) $ is an adaptive triadic network dynamical system such that: 
\begin{enumerate}
    \item the vector field is $C^1$ on $\mathbb{R}^{N+N^2+N^3}\setminus\mathcal{H}_0$, 
    where the union of vanishing hyperplanes is
    $$
    \mathcal{H}_0 := 
    \bigcup_{1\leq i<j\leq N}\{A^{(1)}_{ij}=0\}
    \;\cup\;
    \bigcup_{1\leq i<j<k\leq N}\{A^{(2)}_{ijk}=0\};
    $$
    \item $\mathcal{H}(t)$ converges to the symmetric regime in time $ t_0\in[0, \infty)$ 
    with parameter $0 < \varepsilon \ll \delta$, so that $A^{(1)}$ and $A^{(2)}$ 
    can be regarded as symmetric in their indices;
    \item the system is outward pointing at every triple $(i,j,k)$ with threshold 
    $\delta$  as in Definition \ref{def:asymptotic_regime}.
\end{enumerate}
If at some time $t_1 \geq t_0$ the configuration 
$(A^{(1)}(t_1),A^{(2)}(t_1))$ lies in $\Omega_\delta$, then
$$
(A^{(1)}(t),A^{(2)}(t)) \in \Omega_\delta 
\qquad \text{for all } t \geq t_1,
$$
i.e.\ once simplicial at threshold $\delta$, the configuration remains simplicial 
for all future time.\\

\end{theorem}

\begin{proof}
Let $K_{ijk} := \mathbb{R}^4 \setminus M^{2,\delta}_{ijk}$ and $\Omega_\delta := \prod_{i<j<k} K_{ijk}$. 
Assume $(A^{(1)}(t_1),A^{(2)}(t_1)) \in \Omega_\delta$, but suppose for contradiction that there exists $T > t_1$ with 
$(A^{(1)}(T),A^{(2)}(T)) \notin \Omega_\delta$. Define the first exit time
$$
\tau := \inf\{\, t \geq t_1 : (A^{(1)}(t),A^{(2)}(t)) \notin \Omega_\delta \,\}.
$$
Then for some triple $(i,j,k)$ we have $x^{(ijk)}(t) \in K_{ijk}$ for $t < \tau$ and 
$x^{(ijk)}(\tau) \in \partial M^{2,\delta}_{ijk}$.

By definition of the boundary, $x^{(ijk)}(\tau)$ lies on one of the faces $X_1,X_2, X_3,X_4$ from
\ref{equation: boundary of bad set}.Thus, there exists an outward normal $n \in N_{\partial M}(x^{(ijk)}(\tau))$. 
To enter $M^{2,\delta}_{ijk}$ at time $\tau$, at least one active constraint must decrease, i.e.
$
n \cdot F^{(ijk)}\bigl(x^{(ijk)}(\tau)\bigr) < 0
$
for some such $n$. However, by the outward--pointing condition (Definition~\ref{def: outward pointing symmetric}), we have
$
n \cdot F^{(ijk)}(x) \geq 0 $ for every $ x \in \partial M^{2,\delta}_{ijk}, \; n \in N_{\partial M}(x).$
This yields a contradiction. Hence no exit from $K_{ijk}$ is possible, and therefore 
$x^{(ijk)}(t) \in K_{ijk}$ for all $t \geq t_1$ and every triple $(i,j,k)$. It follows that
$
(A^{(1)}(t),A^{(2)}(t)) \in \Omega_\delta $ for all $t \geq t_1,$
as claimed.
\end{proof}
\vspace{1em}

\begin{remark}\label{remark: interior drift symmetric}
Theorem \ref{thm:retention} is a stability theorem; it requires that a simplicial complex structure is already present on the hypergraph after convergence to the symmetric regime. This, of course, need not necessarily be the case.  In order to get emergence of such a simplicial complex, consider the following.  If we start without a simplicial complex, it is because there exists a degree $2$ simplex $A^{(2)}_{ijk}$  above the $\delta$-threshold, but at least one of the degree $1$ subfaces is below the threshold, i.e $\min\left( |A^{(1)}_{ij}|, |A^{(1)}_{ik}| , |A^{(1)}_{jk}| \right) < \delta$. Therefore, to enter the desired region, we require that the global simplicial region $\Omega_\delta$ is a \emph{global attracting set. }
\end{remark}

\vspace{1em}

\begin{definition}[Global Attracting Set]\label{def: global attractor}
Let $(X,d)$ be a metric space and let $\{\Phi^t\}_{t \geq 0}$ denote a dynamical system on $X$. 
A set $\mathcal{A} \subseteq X$ is called a \emph{global attracting set} if the following conditions hold:
\begin{enumerate}
    \item \textbf{Invariance:} For all time $t \geq 0$, we have that $ \Phi^t(\mathcal{A}) = \mathcal{A}.$
    \item \textbf{Global attraction:} There exists a time $0 \leq t^\prime < \infty$ such that for all $x_0 \in X,$ $\Phi^{t^\prime}(x_0) \in A.$ 
\end{enumerate}
\end{definition}

If we now suppose our vector field satisfies both the outward pointing condition from Definition \ref{def: outward pointing symmetric}, and the set $\Omega_\delta$ is a global attracting set, then we now a simplicial complex will arise at some point.  We can thus formulate the following corollary to Theorem \ref{thm:retention}.\\

\begin{corollary}\label{corollary: emergent simplicial complex symmetric}
Let $(\textbf{x},A^{(1)},A^{(2)})$ be an adaptive triadic network dynamical system that gives rise to a hypergraph $\mathcal{H}(t)$. Assume the hypotheses of Theorem \ref{thm:retention}. Suppose furthermore that the global simplicial region $\Omega^\delta$ is a global attracting set in the sense of Definition \ref{def: global attractor}. Then there exists a time $0 \leq t_0 < \infty$ such that for all $t^\prime \geq t_0,$ $\mathcal{H}(t^\prime)$ is a simplicial complex.
\end{corollary}

\subsection{Examples}

\begin{example}\label{example: symmetric emergence kuramoto type example}
Consider the following example, which satisfies the conditions of Corollary \ref{corollary: emergent simplicial complex symmetric}. We first introduce a non-smooth model,  and then smoothen out using known approximations. The general dynamics follow the template of Example \ref{def:ATKHH}.

$$
\dot\theta_i
   =\omega_i
     +\frac1N\sum_{j=1}^{N}A^{(1)}_{ij}\,
        \sin(\theta_j-\theta_i)
     +\frac1{N^{2}}\sum_{j,k=1}^{N}A^{(2)}_{ijk}\,
        \sin(\theta_j+\theta_k-2\theta_i),
$$
$$
\dot A^{(1)}_{ij}
  =-\alpha\bigl(A^{(1)}_{ij}-\cos(\theta_i-\theta_j)\bigr)
   +\beta\,J_{ij}\,\delta\,
    \text{sgn}_{\!s}\bigl(A^{(1)}_{ij}\bigr),
\qquad
\dot A^{(2)}_{ijk}
  =-\gamma\bigl(A^{(2)}_{ijk}-\delta\cos(\theta_i+\theta_j+\theta_k)\bigr),
$$
where $\alpha,\gamma,\delta, \beta > 0$, and
$$
\text{sgn}_{\!s}(A^{(1)}_{ij})=\text{sgn}\!\left(\frac{A^{(1)}_{ij}+A^{(1)}_{ji}}{2}\right),\qquad
J_{ij}=J_{ji}:=\begin{cases}
1 & \text{if }\min\{|A^{(1)}_{ij}|,|A^{(1)}_{ji}|\}\leq\delta\text{ and } \exists k:\,|A^{(2)}_{ijk}|\geq\delta,\\
0 & \text{otherwise.}
\end{cases}
$$
The above dynamics highlight the simplicial preserving nature of this system, in particular, the edge dynamics positively reinforce whenever they are under the simplicial threshold and a parent face lies above it. A technical obstacle is that the above adaptation rules for edges and hyperedges  
involve nonsmooth maps such as $\max$, $\min$, and a indicator function.  This causes an issue, since  the vector field fails to be $C^1$ off the vanishing hyperplanes $\mathcal{H}_0$, so our outward-pointing arguments may not necessarily hold.  To remedy this, 
we introduce smooth approximations with a small smoothing scale $0 < \zeta \ll 1$. These approximations can be found  
in convex analysis and in some machine learning algorithms, among other areas, \cite{boyd2004convex}. 

\begin{enumerate}
    \item \textbf{Log--sum--exp for $\max$.} For $x=(x_1,\dots,x_N)\in\mathbb{R}^N$ we set
    $$
    \max\nolimits_\zeta(x_1,\dots,x_N)
    :=\zeta\log\!\Big(\sum_{k=1}^N e^{x_k/\zeta}\Big).
    $$
    This is convex and $C^\infty$. As $\zeta\to 0$, it converges uniformly on compacts to $\max_k x_k$. 
    The error is controlled by
    $$
    0 \;\leq\; \max\nolimits_\zeta(x)-\max_k x_k \;\leq\; \zeta\log N.
    $$

    \item \textbf{Negative log--sum--exp for $\min$.} Since $\min(a,b)=-\max(-a,-b)$, we define
    $$
    \min\nolimits_\zeta(a,b):=-\zeta\log\!\big(e^{-a/\zeta}+e^{-b/\zeta}\big).
    $$
    This is smooth and concave, converging to $\min(a,b)$ with bounds
    $$
    \min(a,b)-\zeta\log 2 \;\leq\; \min\nolimits_\zeta(a,b)\;\leq\;\min(a,b).
    $$

    \item \textbf{Smoothed Heaviside and sign.} The step function $H(z)=\mathbf{1}_{\{z>0\}}$ 
    and $\text{sgn}(z)$ are replaced by
    $$
    H_\zeta(z)=\tfrac12\big(1+\tanh(z/\zeta)\big),\qquad 
    \text{sgn}_\zeta(z)=\tanh(z/\zeta).
    $$
    Both converge pointwise as $\zeta\to 0$. Moreover the error away from the switching layer 
    is exponentially small:
    $$
    \big|\text{sgn}_\zeta(z)-\text{sgn}(z)\big|\;\leq\;2e^{-2|z|/\zeta},\qquad
    \big|H_\zeta(z)-H(z)\big|\;\leq\;e^{-2|z|/\zeta}.
    $$

    \item \textbf{Symmetrised sign.} For interactions depending on two arguments one may use 
    the symmetric smoothing
    $$
    \text{sgn}_{s,\zeta}(x,y)=\tanh\!\big((x+y)/(2\zeta)\big),
    $$
    which converges to $\text{sgn}(x)=\text{sgn}(y)$ when both are aligned.
\end{enumerate}
The smoothened dynamics now read as follows, 
$$
\dot\theta_i
=\omega_i
+\frac{1}{N}\sum_{j=1}^{N}A^{(1)}_{ij}\,\sin(\theta_j-\theta_i)
+\frac{1}{N^{2}}\sum_{j,k=1}^{N}A^{(2)}_{ijk}\,\sin(\theta_j+\theta_k-2\theta_i),
$$
$$
\dot A^{(1)}_{ij}
=-\alpha\bigl(A^{(1)}_{ij}-\cos(\theta_i-\theta_j)\bigr)
+\beta\,J^{(\zeta)}_{ij}\,\delta\,
\text{sgn}_{\!s,\zeta}\!\bigl(A^{(1)}_{ij},A^{(1)}_{ji}\bigr),
\qquad
\dot A^{(2)}_{ijk}
=-\gamma\bigl(A^{(2)}_{ijk}-\delta\cos(\theta_i+\theta_j+\theta_k)\bigr),
$$

where $$
J^{(\zeta)}_{ij} := H_\zeta\!\left(\delta-\min_\zeta\big(|A^{(1)}_{ij}|,|A^{(1)}_{ji}|\big)\right)\,
H_\zeta\!\left(\max_\zeta\{|A^{(2)}_{ij1}|,\dots,|A^{(2)}_{ijN}|\}-\delta\right).
$$

By construction, the above dynamics are $C^\infty$ on $\mathbb{R}^{N + N^2 + N^3}\setminus \mathcal{H}_0.$ Now, we show that our dynamics satisfy the hypotheses of Corollary \ref{corollary: emergent simplicial complex symmetric}. First we show it converges to the symmetric regime. Let $A^{(1)}=A^{(1)}_{\text{sym}} + A^{(1)}_{\text{alt}} $ be the orthogonal decomposition into the symmetric and alternating parts, and define
$$
M_{ij}=\cos(\theta_i-\theta_j),\qquad
B_{ij}=\beta\,\delta\,J^{(\zeta)}_{ij}\,\text{sgn}_{s,\zeta}\!\bigl(A^{(1)}_{ij},A^{(1)}_{ji}\bigr).
$$
By construction $M$ is symmetric and $B_{ij}=B_{ji}$, hence $B$ is symmetric. Projecting the edge dynamics gives
$$
\dot A^{(1)}_{\text{sym}} =-\alpha\,(A^{(1)}_{\text{sym}}-M)+B,\qquad \dot A^{(1)}_{\text{alt}}=-\alpha\,A^{(1)}_{\text{alt}},
$$
so $A^{(1)}_{\text{alt}}(t)=e^{-\alpha t}A^{(1)}_{\text{alt}}(0)\to 0$ exponentially. For the triads, splitting
$A^{(2)}=A^{(2)}_{\text{sym}}+A^{(2)}_{\text{alt}}+A^{(2)}_{\text{mix}}$
yields
$$
\dot A^{(2)}_{\text{sym}}=-\gamma\,(A^{(2)}_{\text{sym}}-\delta\cos(\theta_i+\theta_j+\theta_k)),
\quad
\dot A^{(2)}_{\text{alt}}=-\gamma\,A^{(2)}_{\text{alt}},
\quad
\dot A^{(2)}_{\text{mix}}=-\gamma\,A^{(2)}_{\text{mix}},
$$
hence $A^{(2)}_{\text{alt}}(t),A^{(2)}_{\text{mix}}(t)\to 0$. Consequently the system converges to the symmetric regime.

Next, we verify that the dynamics satisfy the outward pointing boundary conditions. 
Let $X_1$ be as in \ref{equation: boundary of bad set}. On $X_1$,
$$
\text{sgn}\!\big(A^{(2)}_{ijk}\big)\,\dot A^{(2)}_{ijk}
=-\gamma\Big(|A^{(2)}_{ijk}|-\delta\,\text{sgn}\!\big(A^{(2)}_{ijk}\big)\cos(\theta_i+\theta_j+\theta_k)\Big)
\le 0,
$$
since $|A^{(2)}_{ijk}|=\delta$ and $|\cos|\le 1$.

For the binary boundary $X_2$, we check after the system has entered the symmetric regime, so $A^{(1)}_{ij}=A^{(1)}_{ji}$. Then
$$
\text{sgn}\!\big(A^{(1)}_{ij}\big)\,\text{sgn}_{s,\zeta}\!\big(A^{(1)}_{ij},A^{(1)}_{ji}\big)
=\tanh\!\left(\tfrac{|A^{(1)}_{ij}|}{\zeta}\right)\ge \tanh\!\left(\tfrac{\delta}{\zeta}\right)=:m_\zeta\in(0,1),
$$
and, because on $X_2$ we have $|A^{(1)}_{ij}|=\delta$ and there exists $ k$ such that $|A^{(2)}_{ijk}|\ge\delta$,
$$
J^{(\zeta)}_{ij}\ge H_\zeta(0)\,H_\zeta(0)=\tfrac14 =: \underline J_\zeta.
$$
Therefore, on $X_2$,
$$
\text{sgn}\!\big(A^{(1)}_{ij}\big)\,\dot A^{(1)}_{ij}
\;\ge\; -\alpha(\delta+1)+\beta\,\delta\,\underline J_\zeta\,m_\zeta.
$$
Choosing

$$
\beta  >   \frac{\alpha(\delta+1)}{\delta\,\underline J_\zeta\,m_\zeta}
\;=\; \frac{4\alpha(\delta+1)}{\delta\,\tanh(\delta/\zeta)}
$$

ensures $\text{sgn}(A^{(1)}_{ij})\,\dot A^{(1)}_{ij}\ge 0$ on $X_2$; the same argument applies to $X_3$ and $X_4$.

Thus our dynamics satisfy the hypotheses of Theorem \ref{thm:retention} and thus retain a simplicial complex structure if one is present after convergence to the symmetric regime. To show one will emerge, we must show that the set $\Omega_\delta$ is a global attracting set for the dynamical system. 

To this end, fix a triple $(i,j,k)$ with $|A^{(2)}_{ijk}|\ge \delta$ and suppose some edge, say $|A^{(1)}_{ij}|<\delta$, violates downward closure. Then
\[
J^{(\zeta)}_{ij}\ \ge\ \tfrac14,\qquad 
\text{sgn}\!\left(A^{(1)}_{ij}\right)\text{sgn}_{\zeta}\!\left(A^{(1)}_{ij},A^{(1)}_{ji}\right)
=\tanh\!\Big(\tfrac{|A^{(1)}_{ij}|}{\zeta}\Big)\ \ge\ m_\zeta:=\tanh(\delta/\zeta).
\]
Hence
\[
\text{sgn}\!\left(A^{(1)}_{ij}\right)\dot A^{(1)}_{ij}
\ \ge\ -\alpha(\delta+1)+\beta\,\frac{\delta m_\zeta}{4}.
\]
By the parameter condition, the right-hand side is bounded below by some $\eta>0$. Thus every violating edge under an active parent triad increases in magnitude with uniform speed until it reaches $\delta$, with hitting time bounded by
\[
t_{\mathrm{hit}}\ \le\ \frac{\delta-|A^{(1)}_{ij}(0)|}{\eta}\ <\infty.
\]

We numerically simulate these dynamics to illustrate the emergence of a simplicial complex structure and the system's convergence to the symmetric regime. The system was integrated over $t \in [0,25]$ with $250$ evaluation points. The network consisted of $N=4$ oscillators with natural frequencies $\omega_i \sim \mathcal{N}(0,0.5^2)$, initial phases drawn uniformly from $[0,2\pi)$, and random initial binary and triadic weights in $(-0.25,0.25)$. To test the closure property, we introduced a single simplicial violation by setting one edge $(0,1)$ to be weak while its parent triad $(0,1,2)$ was strong. The remaining parameters were chosen as $\alpha=0.5$, $\beta=25.0$, $\gamma=0.8$, $\delta=0.5$, and $\zeta=0.05$, which satisfy the analytical conditions for simplicial emergence. The results (Figure~\ref{fig:transient_simplicial_plots}) confirm that all non-symmetric components decay exponentially, and a simplicial complex structure is observed. The full implementation and initial conditions are detailed in the code repository.


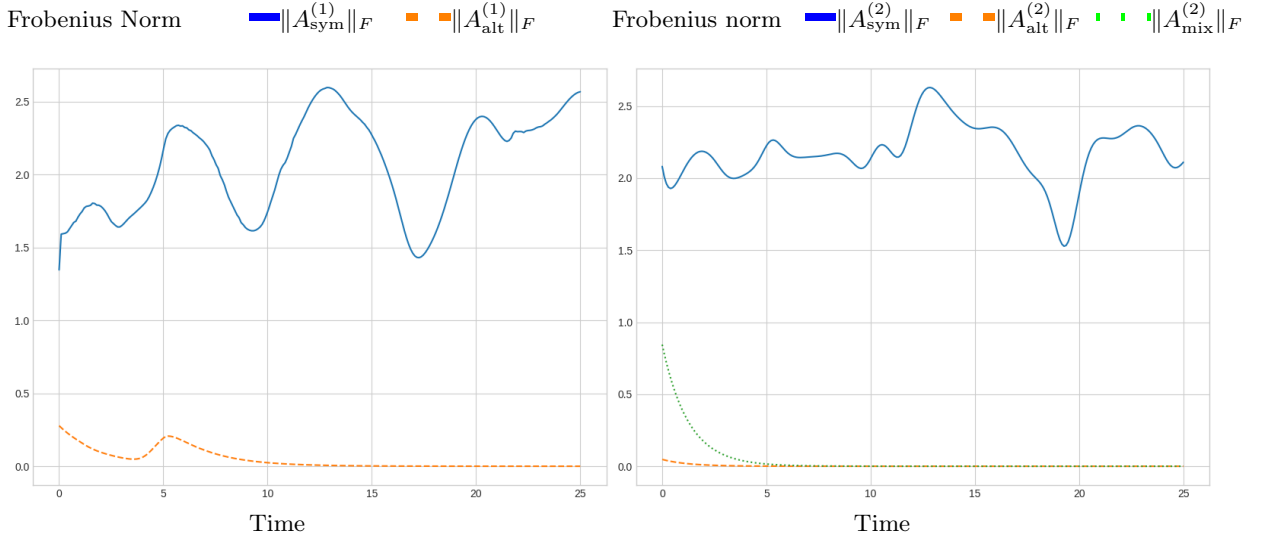
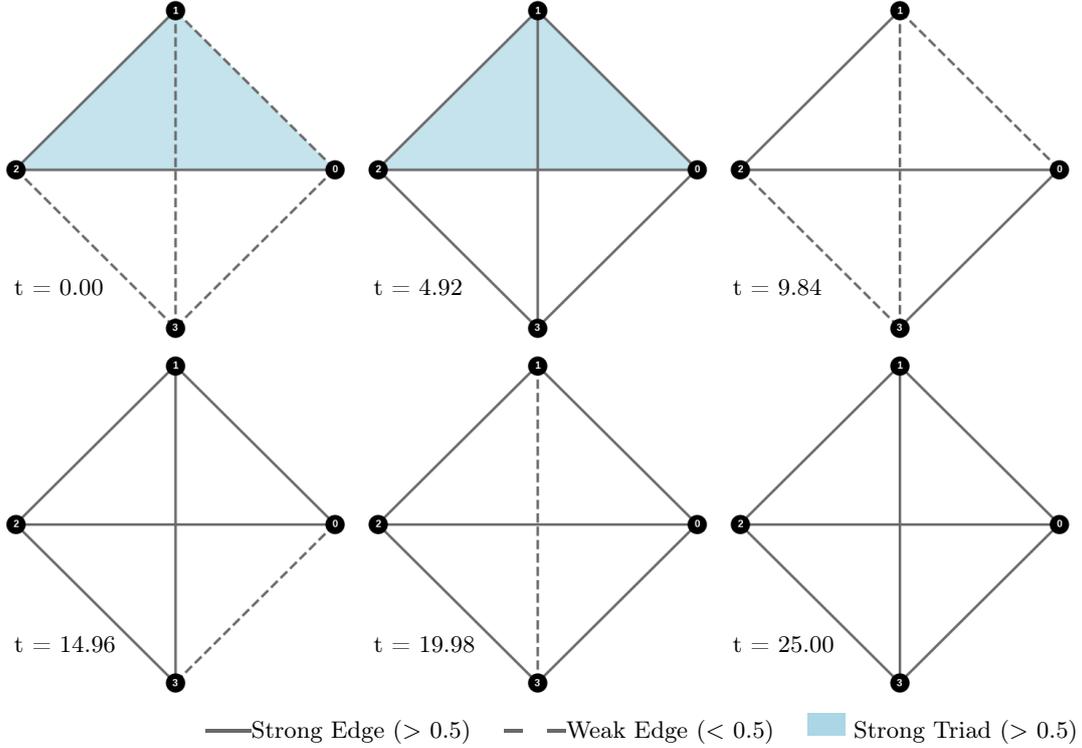
\begin{figure}[h!]
    \centering
    
    \begin{subfigure}[b]{\textwidth}
        \centering
        \begin{overpic}[width=\textwidth] 
            {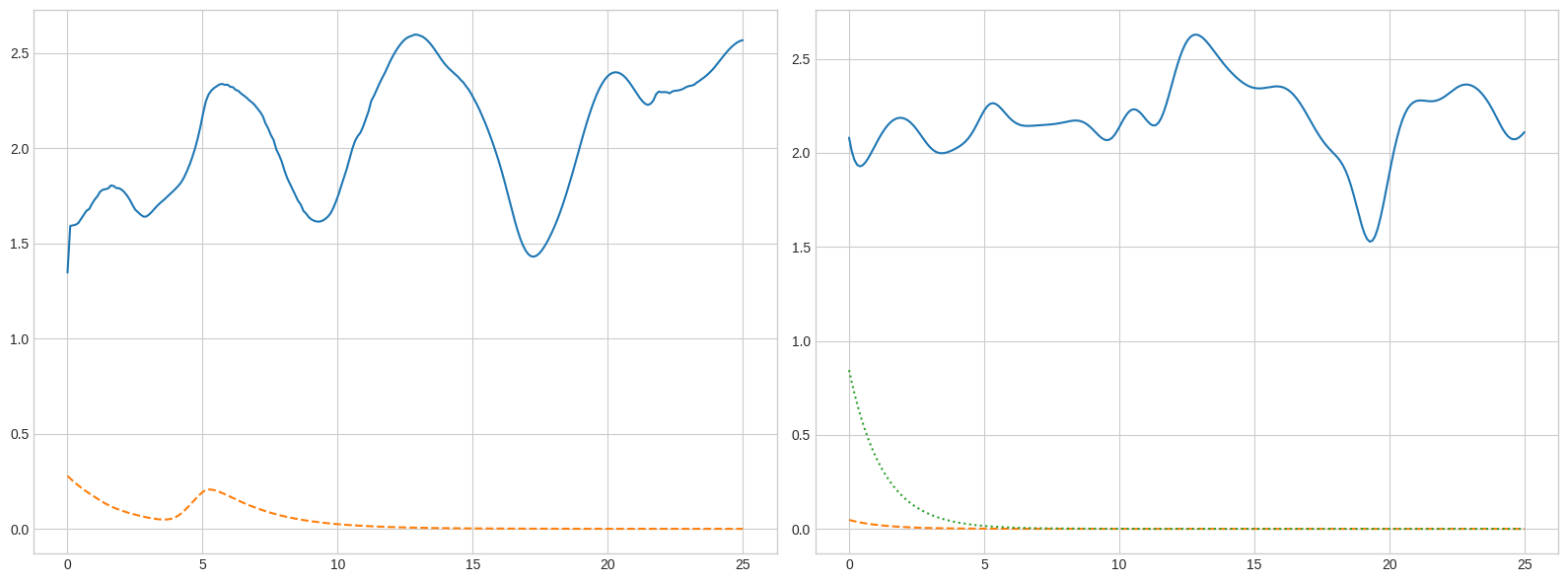}
            
            \put(0, 40){\small Frobenius Norm } 
            \put(20, -2){\small Time} 
            \put(15, 96){\small \textbf{Binary Layer (Edges)}}

            \put(20, 40){\legendline{blue}{$\Vert A^{(1)}_{\text{sym}}\Vert_F$}}
            \put(33, 40){\legenddashed{orange}{$\Vert A^{(1)}_{\text{alt}}\Vert_F$}}
            
            \put(50, 40){\small Frobenius norm} 
            \put(70, -2){\small Time} 
            \put(63, 96){\small \textbf{Triadic Layer (Triangles)}}

            \put(66, 40){\legendline{blue}{$\Vert A^{(2)}_{\text{sym}}\Vert_F$}}
            \put(78, 40){\legenddashed{orange}{$\Vert A^{(2)}_{\text{alt}}\Vert_F$}}
            \put(90, 40){\legenddotted{green}{$\Vert A^{(2)}_{\text{mix}}\Vert_F$}}
            
        \end{overpic}
        
        \vspace{1em}
        \caption{Time series of the Frobenius norms for the isotypic components of the binary tensor $A^{(1)}$ (left) and the triadic tensor $A^{(2)}$ (right). The norms of the antisymmetric part ($\Vert A^{(1)}_{\text{alt}}\Vert_F$) and the antisymmetric and mixed parts of the triad ($\Vert A^{(2)}_{\text{alt}}\Vert_F$, $\Vert A^{(2)}_{\text{mix}}\Vert_F$) decay exponentially to zero, confirming convergence to the symmetric regime.}
        \label{fig:convergence_plots}
    \end{subfigure}
    
    \vspace{2em} 
    
    \begin{subfigure}[b]{\textwidth}
        \centering
        \begin{overpic}[width=0.9\textwidth] 
            {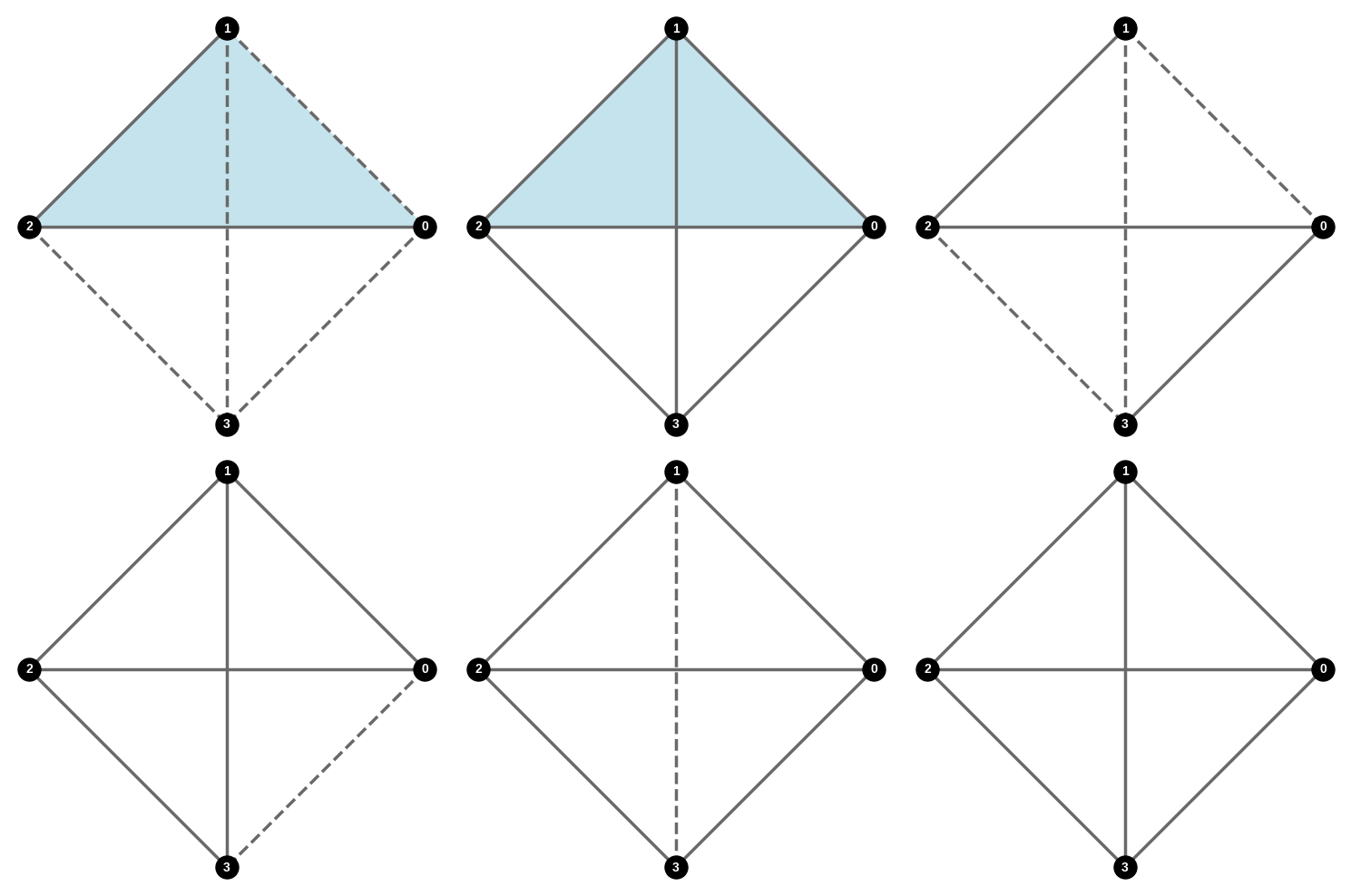} 
            
            \put(2, 38){\small t = 0.00}
            \put(35, 38){\small t = 4.92}
            \put(68, 38){\small t = 9.84}
            
            \put(2, 5){\small t = 14.96}
            \put(35, 5){\small t = 19.98}
            \put(68, 5){\small t = 25.00}

            \put(18, -3){
                \small
                \definecolor{strongtriad}{rgb}{0.67, 0.84, 0.9} 
                
                \color{dimgray}\rule[0.5ex]{6mm}{1.5pt}\color{black} Strong Edge ($>$ 0.5) \quad
                \color{dimgray}\rule[0.5ex]{2.5mm}{1.5pt} \rule[0.5ex]{1mm}{0mm} \rule[0.5ex]{2.5mm}{1.5pt}\color{black} Weak Edge ($<$ 0.5) \quad
                \color{strongtriad}\rule[0.2ex]{5mm}{3mm} \color{black} Strong Triad ($>$ 0.5)
            }
        \end{overpic}
        
        \vspace{2em} 
        \caption{Snapshots of the network structure at different times. The system starts with a simplicial complex violation at $t=0$, where the triad (0,1,2) is strong (blue shading) but its edge (0,1) is weak (dashed line). The adaptive dynamics reinforce the weak edge, restoring the downward closure property. A persistent triad is observed before decaying.}
        \label{fig:simplicial_evolution}
    \end{subfigure}

    \caption[Simulation of simplicial emergence (transient case).]{Simulation of simplicial emergence (transient case). (a) Frobenius norms converging to the symmetric regime. (b) Network snapshots showing restoration of downward closure.}
    \label{fig:transient_simplicial_plots}
\end{figure}

\end{example}

The previous example demonstrates that, although a simplicial complex structure is preserved by the dynamics, this structure can become trivial, effectively collapsing onto a $1$–simplicial complex, i.e. a graph. In this regime, the higher–order components decay and the network behaves as if only pairwise interactions were present. This observation aligns with much of the existing network science literature over the past decades, where higher–order effects are often neglected: even when such terms are present in the governing equations, they frequently vanish asymptotically and leave no persistent signature in the higher–order structure.  Such a phenomenon was explored in \cite{horstmeyer2020adaptive}, where the authors studied an adaptive voter model explicitly defined on a simplicial complex. This framework included not just pairwise interactions on edges (1-simplices) but also group peer-pressure effects on filled-in triangles (2-simplices). They found that the adaptive rewiring dynamics led to a depletion of the higher-order structures. The model exhibited a multiscale hierarchy where the higher-order 2-simplices die out on a faster time scale, vanishing before the dynamics on the edges had depleted . This sequential decay effectively collapses the higher-order network onto a simple graph, after which the system evolves as a standard pairwise adaptive model. This result provides a direct justification for treating certain adaptive systems as effectively pairwise, as the higher-order components are shown to be transient and vanish dynamically.

In contrast, the next example illustrates a qualitatively different behaviour, one in which higher–order interactions persist in a non-trivial and structured manner. In this case, the system maintains a subset of $2$–simplices over time without filling in all possible higher–order connections, yielding a genuinely simplicial, but not fully saturated, complex.\\

\begin{example} 
    Consider the following dissipative, consensus based,  adaptive, triadic network dynamical system, defined on $N$ nodes each with a state $x_i(t) \in \mathbb{R}.$
    \begin{align*}
\dot{x}_i &= \frac{1}{N}\sum_{j=1}^{N} A_{ij}^{(1)}(x_j - x_i) + \frac{1}{N^2}\sum_{j,k=1}^{N} A_{ijk}^{(2)} \left( \frac{x_j + x_k}{2} - x_i \right), \\
\dot{A}_{ij}^{(1)} &= -\alpha \left( A_{ij}^{(1)} - \kappa_1 e^{-\lambda_1(x_i - x_j)^2} \right) + \beta \delta J_{ij}^{(\zeta)} \text{sgn}_{s,\zeta}(A_{ij}^{(1)}, A_{ji}^{(1)}), \\
\dot{A}_{ijk}^{(2)} &= -\gamma \left( A_{ijk}^{(2)} - \kappa_2 e^{-\lambda_2 V_{ijk}} \right),
\end{align*}
where $V_{ijk} = \frac{1}{6} \sum_{p,q \in \{i,j,k\}} (x_p-x_q)^2$ and $J_{ij}^{(\zeta)}$ and $\text{sgn}_{s, \zeta}$ are the same as in Example~\ref{example: symmetric emergence kuramoto type example}, $0 < \zeta \ll 1$ is a smoothing parameter.

We make the following observations about the dynamics. 

\begin{itemize}
  \item \emph{Nodes.} The terms $A^{(1)}_{ij}(x_j-x_i)$ implement pairwise diffusion. The triadic term $A^{(2)}_{ijk}\big(\frac{x_j+x_k}{2}-x_i\big)$ is a higher-order consensus force pulling node $i$ towards the average of $j$ and $k$.
  \item \emph{Triads.} $V_{ijk}$ measures the local variance on $\{i,j,k\}$ and is symmetric in $(i,j,k)$. Thus $A^{(2)}_{ijk}$ relaxes towards the high target $\kappa_2 \gg 0$ whenever the triple is nearly in consensus ($V_{ijk}\approx 0$), creating a positive feedback that sustains strong triads near consensus.
  \item \emph{Edges.} $A^{(1)}_{ij}$ follows a relaxation to the symmetric target $\kappa_1 e^{-\lambda_1(x_i-x_j)^2}$, while the $J^{(\zeta)}_{ij}$–term enforces simplicial closure by reinforcing weak edges that lie under a strong parent triad.
\end{itemize}

We now verify that the above satisfies all assumptions of Corollary \ref{corollary: emergent simplicial complex symmetric}.
\begin{enumerate}
    \item[\textbf{(i)}] \textit{Smoothness.}
    The right-hand side consists of compositions of smooth functions
    ($e^{-z^2}$, smoothed sign $\text{sgn}_{s,\zeta}$, and $J^{(\zeta)}_{ij}$ built from $H_\zeta$ and $\min_\zeta$),
    hence the vector field
$(\dot{ \textbf{x}}, \dot A^{(1)}, \dot A^{(2)})$
    is $C^\infty$ on $\mathbb{R}^{N + N^2 + N^3}\setminus H_0$.
    Therefore, the regularity hypothesis of Theorem \ref{thm:retention} (and thus of Corollary\ref{corollary: emergent simplicial complex symmetric}) holds.

    \item[\textbf{(ii)}] \textit{Convergence to the symmetric regime.}
    The evolution laws for both layers drive antisymmetric and mixed components to zero.
    Writing $A^{(1)} = A^{(1)}_{\mathrm{sym}} + A^{(1)}_{\mathrm{alt}}$ and
    $A^{(2)} = A^{(2)}_{\mathrm{sym}} + A^{(2)}_{\mathrm{alt}} + A^{(2)}_{\mathrm{mix}}$,
    we obtain
    \begin{align*}
    \dot A^{(1)}_{\mathrm{sym}} &= -\alpha\!\left(A^{(1)}_{\mathrm{sym}} - M\right) + B,
    &\dot A^{(1)}_{\mathrm{alt}} &= -\alpha A^{(1)}_{\mathrm{alt}},\\[3pt]
    \dot A^{(2)}_{\mathrm{sym}} &= -\gamma\!\left(A^{(2)}_{\mathrm{sym}} - \kappa_2 e^{-\lambda_2 V}\right),
    &\dot A^{(2)}_{\mathrm{alt}} &= -\gamma A^{(2)}_{\mathrm{alt}},
    &\dot A^{(2)}_{\mathrm{mix}} &= -\gamma A^{(2)}_{\mathrm{mix}},
    \end{align*}
    where $M_{ij} = \kappa_1 e^{-\lambda_1(x_i - x_j)^2}$ and $B_{ij} = \beta\delta J^{(\zeta)}_{ij}\text{sgn}_{s,\zeta}(A^{(1)}_{ij},A^{(1)}_{ji})$
    are symmetric. \\
    Hence $A^{(1)}_{\mathrm{alt}}(t), A^{(2)}_{\mathrm{alt}}(t), A^{(2)}_{\mathrm{mix}}(t) \to 0$ exponentially,
    and the system converges to the symmetric regime.

    \item[\textbf{(iii)}] \textit{Outward–pointing condition.}
    On the triadic boundary $X_1 = \{|A^{(2)}_{ijk}| = \delta\}$ we have
    \[
        \text{sgn}(A^{(2)}_{ijk}) \dot A^{(2)}_{ijk}
        = -\gamma\!\left(|A^{(2)}_{ijk}| - \kappa_2 e^{-\lambda_2 V_{ijk}}\text{sgn}(A^{(2)}_{ijk})\right)
        \le 0,
    \]
    since $|A^{(2)}_{ijk}| = \delta$ and $e^{-\lambda_2 V_{ijk}} \le 1$.
    On the binary boundary $X_2 = \{|A^{(1)}_{ij}| = \delta\}$ (after the system has reached symmetry),
    we obtain
    \[
        \text{sgn}(A^{(1)}_{ij}) \dot A^{(1)}_{ij}
        = -\alpha\!\left(|A^{(1)}_{ij}| - \kappa_1 e^{-\lambda_1 (x_i - x_j)^2}\right)
        + \beta\delta\,J^{(\zeta)}_{ij} \ge 0,
    \]
    provided $\kappa_1 \ge \delta$ and $J^{(\zeta)}_{ij}\ge 0$,
    ensuring the flow is directed inward at all boundary faces.
    Thus the outward–pointing condition of Definition~3.6 is satisfied.

    \item[\textbf{(iv)}] \textit{Global attraction of $\Omega_\delta$.}
    The adaptation dynamics ensure that whenever a simplicial violation occurs,
    i.e.\ a triad above threshold sits above a weak edge,
    the $J^{(\zeta)}_{ij}$ term activates to reinforce the edge until closure is restored.
    Hence $\Omega_\delta$ is globally attractive.
\end{enumerate}
Therefore all hypotheses of Corollary \ref{corollary: emergent simplicial complex symmetric} hold, and there exists
$t_0 < \infty$ such that for all $t' \ge t_0$, the hypergraph $\mathcal{H}(t')$ forms a simplicial complex at threshold $\delta$.

We numerically simulate these dynamics to illustrate simplicial emergence with persistence. The simulation uses $N=4$ nodes with initial conditions $x_0 = (0.10,\,0.15,\,0.20,\,2.00)$ and initial weights $\sim U(-0.25,0.25)$. A single simplicial violation was set on the (0,1,2) triad (using symmetric entries): $A^{(2)}_{012} = 0.8$, $A^{(1)}_{02} = 0.6$, $A^{(1)}_{12} = 0.7$, and the weak edge $A^{(1)}_{01} = 0.1$. The model parameters are $\zeta = 0.05$, $\alpha = 0.5$, $\beta = 25.0$, $\gamma = 0.8$, $\delta = 0.5$, $\kappa_1 = 1.0$, $\kappa_2 = 1.2$, $\lambda_1 = 2.0$, and $\lambda_2 = 5.0$. The results confirm convergence to the symmetric regime (Figure \ref{fig:persistent_convergence_plots}) and show the restoration of downward closure with a persistent triad (Figure \ref{fig:persistent_simplicial_evolution}). The full implementation for this example is available in the code repository.
\vspace{1em}

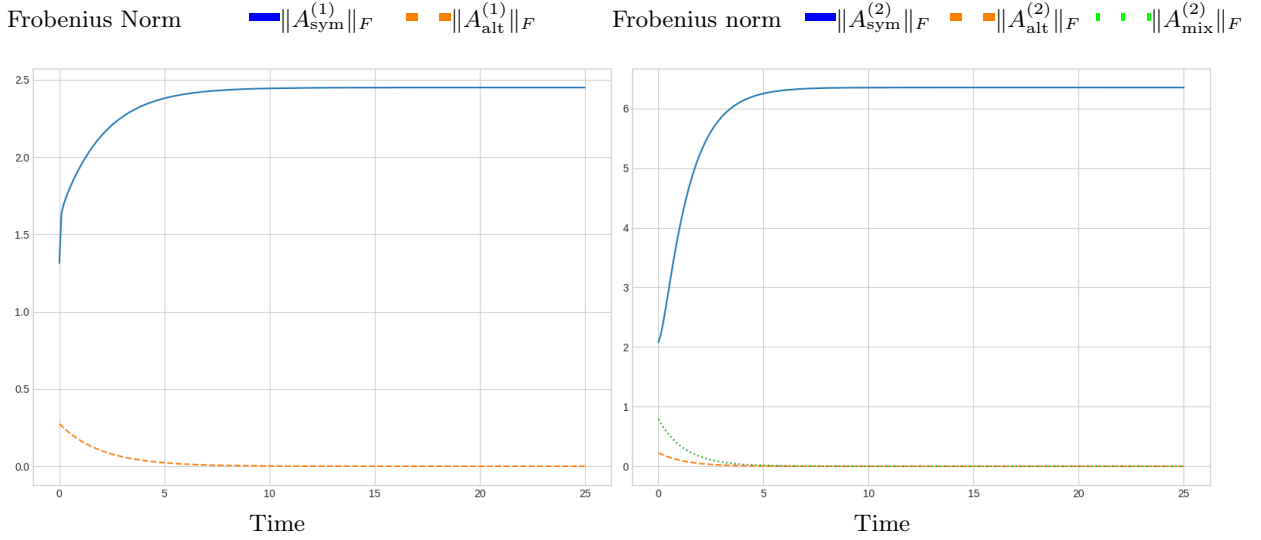
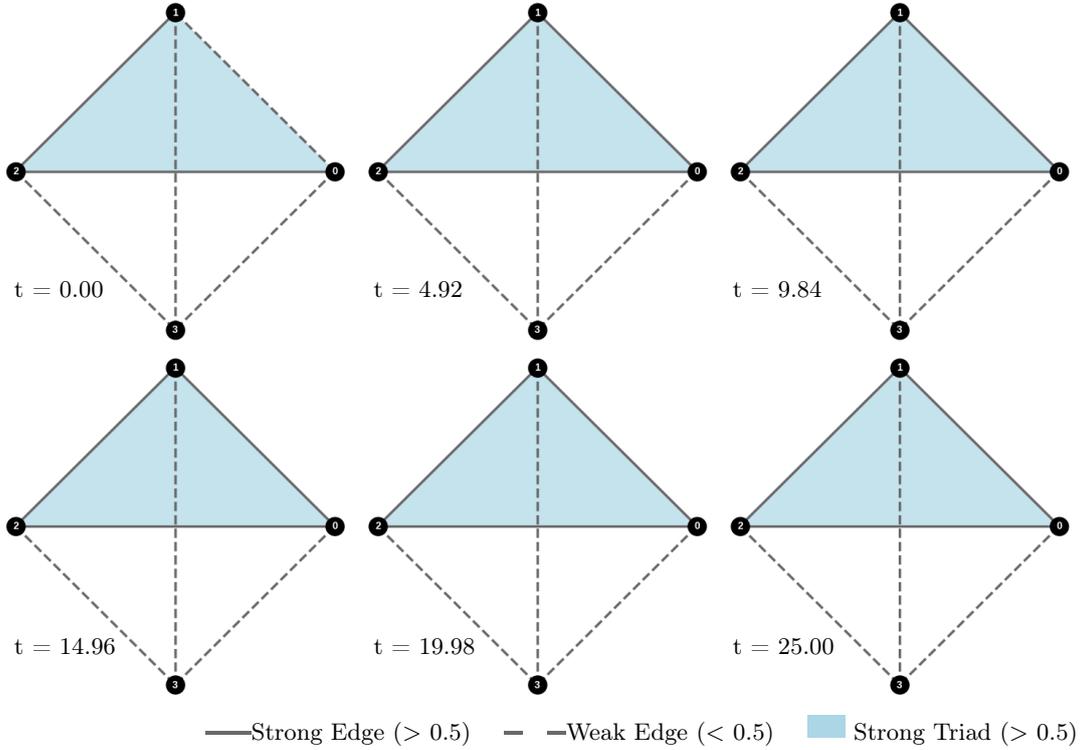
\begin{figure}[h!]
    \centering

    \begin{subfigure}[b]{\textwidth}
        \centering
        \begin{overpic}[width=\textwidth] 
            {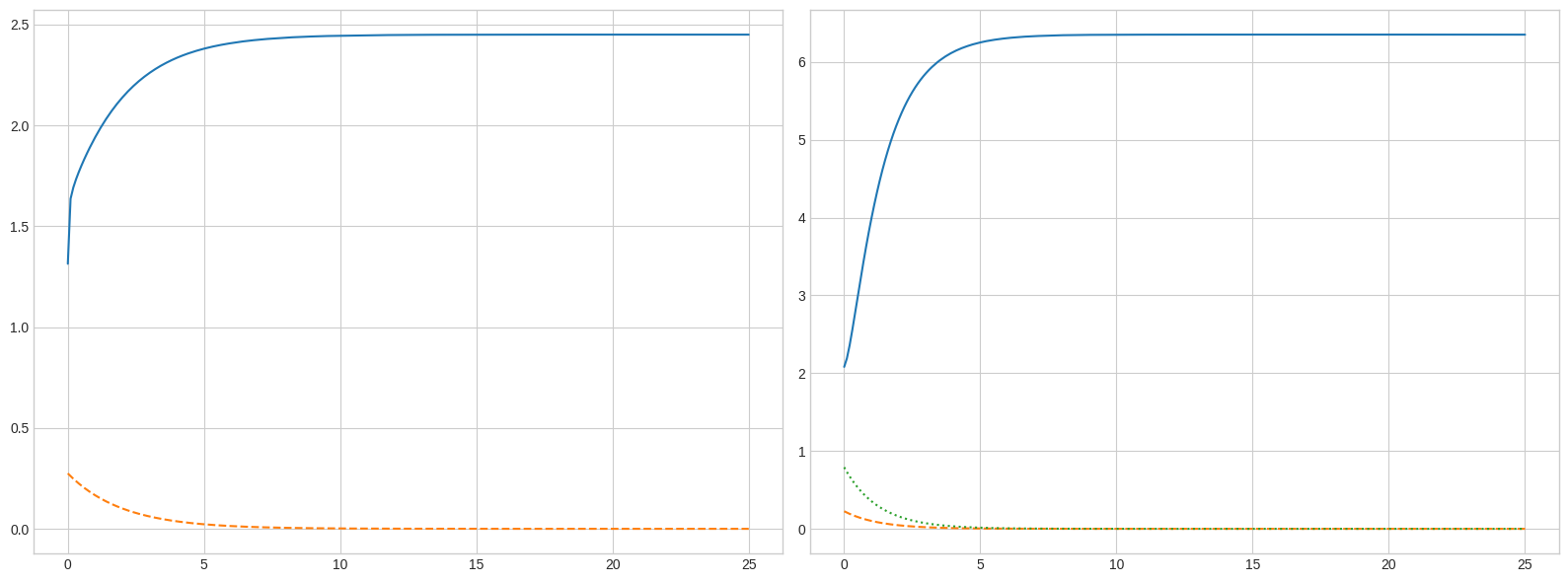}

            \put(0, 40){\small Frobenius Norm } 
            \put(20, -2){\small Time} 
            \put(15, 96){\small \textbf{Binary Layer (Edges)}}

            \put(20, 40){\legendline{blue}{$\Vert A^{(1)}_{\text{sym}}\Vert_F$}}
            \put(33, 40){\legenddashed{orange}{$\Vert A^{(1)}_{\text{alt}}\Vert_F$}}

            \put(50, 40){\small Frobenius norm} 
            \put(70, -2){\small Time} 
            \put(63, 96){\small \textbf{Triadic Layer (Triangles)}}

            \put(66, 40){\legendline{blue}{$\Vert A^{(2)}_{\text{sym}}\Vert_F$}}
            \put(78, 40){\legenddashed{orange}{$\Vert A^{(2)}_{\text{alt}}\Vert_F$}}
            \put(90, 40){\legenddotted{green}{$\Vert A^{(2)}_{\text{mix}}\Vert_F$}}
            
        \end{overpic}
        
        \vspace{1em}
        \caption{Time series of the Frobenius norms for the isotypic components of the binary tensor $A^{(1)}$ (left) and the triadic tensor $A^{(2)}$ (right). The norms of the antisymmetric part ($\Vert A^{(1)}_{\text{alt}}\Vert_F$) and the antisymmetric and mixed parts of the triad ($\Vert A^{(2)}_{\text{alt}}\Vert_F$, $\Vert A^{(2)}_{\text{mix}}\Vert_F$) decay exponentially to zero, confirming convergence to the symmetric regime.}
        \label{fig:persistent_convergence_plots} 
    \end{subfigure}
    
    \vspace{2em}

    \begin{subfigure}[b]{\textwidth}
        \centering
        \begin{overpic}[width=0.9\textwidth] 
            {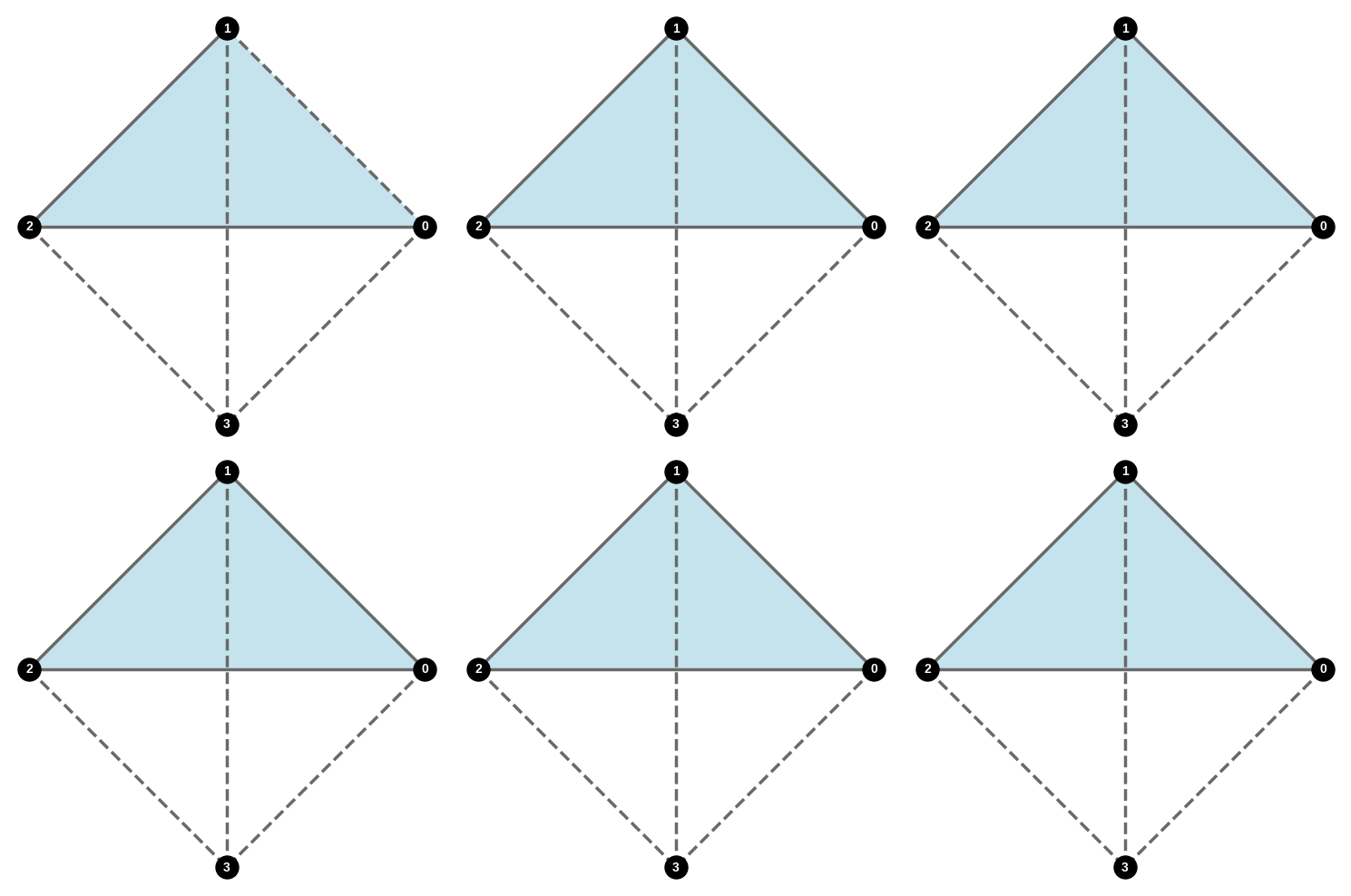}

            \put(2, 38){\small t = 0.00}
            \put(35, 38){\small t = 4.92}
            \put(68, 38){\small t = 9.84}
            
            \put(2, 5){\small t = 14.96}
            \put(35, 5){\small t = 19.98}
            \put(68, 5){\small t = 25.00}

            \put(18, -3){
                \small
                \definecolor{strongtriad}{rgb}{0.67, 0.84, 0.9} 
                
                \color{dimgray}\rule[0.5ex]{6mm}{1.5pt}\color{black} Strong Edge ($>$ 0.5) \quad
                \color{dimgray}\rule[0.5ex]{2.5mm}{1.5pt} \rule[0.5ex]{1mm}{0mm} \rule[0.5ex]{2.5mm}{1.5pt}\color{black} Weak Edge ($<$ 0.5) \quad
                \color{strongtriad}\rule[0.2ex]{5mm}{3mm} \color{black} Strong Triad ($>$ 0.5)
            }
        \end{overpic}
        
        \vspace{2em} 
        
        \caption{Snapshots of the network structure at uniformly spaced times. The system is initialised with a simplicial violation on $(0,1,2)$ (strong triad, weak edge $(0,1)$). The adaptive dynamics reinforce the weak edge via the closure term, restoring downward closure. A nontrivial persistent triad is observed.}
        \label{fig:persistent_simplicial_evolution} 
    \end{subfigure}

    \caption[Simulation of simplicial emergence with persistence.]{Simulation of simplicial emergence with persistence. (a) Frobenius norms converging to the symmetric regime. (b) Network snapshots showing restoration of downward closure.}
    \label{fig:persistent_simplicial_plots} 
\end{figure}

\end{example}

%% file: Antisymm_and_mixed.tex
\section{Oriented Simplicial Complexes and Semi-Simplicial Sets}

\subsection{Antisymmetric Regime}
We now adapt the $\delta$–threshold philosophy of the previous section to the antisymmetric regime, where every binary interaction is oriented and every triadic interaction is alternating.
Suppose our hypergraph $\mathcal{H}(t)$ converges to the antisymmetric regime. Thus, for all $\varepsilon > 0$ there exists a time $t_0$ such that for all $t \geq t_0$ we have that all the symmetric and mixed parts decay to zero, and both $|A^{(1)}_{ij}(t) + A^{(1)}_{ji}(t)| < \varepsilon$ and $|A^{(2)}_{ijk}(t) - \text{sgn}(\sigma)A^{(1)}_{\sigma(i)\sigma(j)\sigma(k)}(t)| < \varepsilon$ for all $\sigma \in S_3$. In such a case, we create equivalence classes and view $A^{(1)}_{ij}$ and $A^{(1)}_{ji}$ as the same object, but differing by a sign,  i.e. for all $t \geq t_0$, $A^{(1)}_{ij} = - A^{(1)}_{ji}.$ Similarly, $A^{(2)}_{ijk}$ will be the representative for all even permutations of its indices, and it is the negative of any permutation by any odd permutation.   \\

\begin{definition}\label{def: oriented downward closure}
    We say an oriented triadic connection strength $A^{(2)}_{ijk}$ respects the oriented downward closure property with parameter $\delta > 0$ if and only if  
    \begin{itemize}
        \item $|A^{(2)}_{ijk}| \geq \delta \implies |A^{(1)}_{ij}|  \geq \delta, |A^{(1)}_{ik}| \geq \delta, |A^{(1)}_{jk}| \geq \delta.$
        \item The orientation of the triadic and binary connections agree, i.e.  \begin{equation*}\label{equation: orientation matching equation}
            \text{sgn}\left( A^{(2)}_{ijk} \right) =\text{sgn}\left(A^{(1)}_{ij}\right) = \text{sgn}\left(A^{(1)}_{jk}\right) = \text{sgn}\left(A^{(1)}_{ik}\right).
        \end{equation*}
    \end{itemize}
\end{definition}
We say an oriented $2$-simplex is present on the nodes $i,j$ and $k$ with parameter $\delta$ if $|A^{(2)}_{ijk}| \geq \delta$ and $A^{(2)}_{ijk}$ respects the oriented downward closure property. Similar to the unoriented case, we say an oriented $1-$simplex is present on the nodes $i$ and $j$, with parameter $\delta,$ if $|A^{(1)}_{ij}| \geq \delta$ and $\text{sgn}(A^{(1)}_{ij}) = -\text{sgn}(A^{(1)}_{ji}).$ The vertices are assumed vacuously present, thus the oriented $1$-simplices respect the downward closure property automatically. 
\tikzset{
  mid arrow/.style={
    postaction={decorate},
    decoration={
      markings,
      mark=at position 0.5 with {\arrow{>}}
    }
  }
}

\begin{minipage}{0.45\textwidth}
  \centering
  \begin{tikzpicture}[scale=5,>=latex,line width=1pt]
    \coordinate (A) at (0,0);
    \coordinate (B) at (1,0);
    \coordinate (C) at (0.5,{sqrt(3)/2});
    \coordinate (M) at ($(A)!0.5!(B)$);
    \coordinate (G) at ($(C)!0.6667!(M)$);
    \fill[gray!20] (A) -- (B) -- (C) -- cycle;
    \draw[blue,mid arrow] (A) -- (B) node[midway,below] {$[i,j]$};
    \draw[blue,mid arrow] (B) -- (C) node[midway,right] {$[j,k]$};
    \draw[blue,mid arrow] (C) -- (A) node[midway,left]  {$[k,i]$};
    \fill[black] (A) circle (0.8pt) node[below left]  {$i$};
    \fill[black] (B) circle (0.8pt) node[below right] {$j$};
    \fill[black] (C) circle (0.8pt) node[above,yshift=1pt] {$k$};
    \draw[blue,->]
      (G)+(45:0.15)
        arc[start angle=45,end angle=315,radius=0.15];
    \node[blue] at (G) {$[i,j,k]$};
  \end{tikzpicture}
  
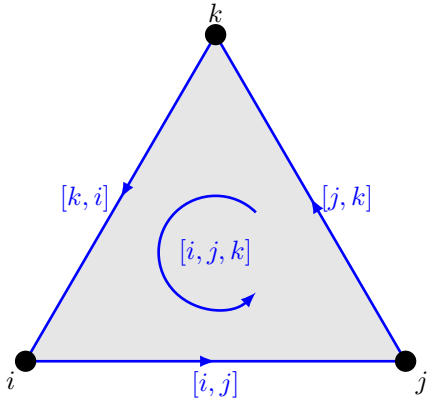
\captionof{figure}{Positive orientation}
\end{minipage}
\hfill
\begin{minipage}{0.45\textwidth}
  \centering
  \begin{tikzpicture}[scale=5,>=latex,line width=1pt]
    \coordinate (A) at (0,0);
    \coordinate (B) at (1,0);
    \coordinate (C) at (0.5,{sqrt(3)/2});
    \coordinate (M) at ($(A)!0.5!(B)$);
    \coordinate (G) at ($(C)!0.6667!(M)$);
    \fill[gray!20] (A) -- (B) -- (C) -- cycle;
    \draw[blue,mid arrow] (B) -- (A) node[midway,below] {$[j,i]$};
    \draw[blue,mid arrow] (C) -- (B) node[midway,right] {$[k,j]$};
    \draw[blue,mid arrow] (A) -- (C) node[midway,left]  {$[i,k]$};
    \fill[black] (A) circle (0.8pt) node[below left]  {$i$};
    \fill[black] (B) circle (0.8pt) node[below right] {$j$};
    \fill[black] (C) circle (0.8pt) node[above,yshift=1pt] {$k$};
    \draw[blue,->]
      (G)+(45:0.15)
        arc[start angle=45,delta angle=-270,radius=0.15];
    \node[blue] at (G) {$[i,k,j]$};
  \end{tikzpicture}
  
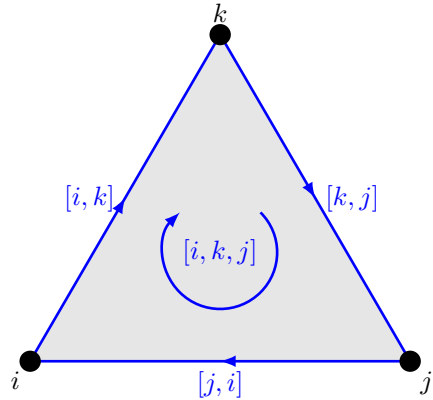
\captionof{figure}{Negative orientation}
\end{minipage}
\vspace{0.5cm}

\begin{definition}
Given an adaptive, triadic, network dynamical system  $\mathcal{H}(t)$, encoded by a matrix $A^{(1)} \in M_N(\mathbb{R})$ and rank-$3$ tensor $A^{(2)} \in \left( \mathbb{R}^N \right)^{\otimes 3}$,  that converges to the antisymmetric regime, we say $\mathcal{H}(t)$ is an oriented simplicial complex with parameter $\delta > 0$, if for all $i,j,k \in \{ 1 , \dots , N\}$, the oriented downward closure property \ref{def: oriented downward closure} is not violated. \\
\end{definition}

Now we want to replicate our results from the symmetric hypergraph case to this oriented case. We begin with the following definition. \\

\begin{definition}[Oriented simplicial region at threshold $\delta$]
For a triple $(i,j,k)$ write $\sigma_{ijk}:=\text{sgn}\!\left(A^{(2)}_{ijk}\right)\in\{-1,0,1\}$ and define the
oriented edge coordinates as
$$
y_{ij}:=\sigma_{ijk}\,A^{(1)}_{ij},\qquad
y_{ik}:=\sigma_{ijk}\,A^{(1)}_{ik},\qquad
y_{jk}:=\sigma_{ijk}\,A^{(1)}_{jk}.
$$
The oriented bad set, i.e. the set that violates the oriented downward closure, is then given by 
$$
M^{2,\delta,\mathrm{or}}_{ijk}
:=\Big\{ (A^{(2)}_{ijk},y_{ij},y_{ik},y_{jk}) \in \mathbb{R}^4 :  |A^{(2)}_{ijk}|\geq\delta\ \text{ and }\ \min(y_{ij},y_{ik},y_{jk})<\delta\Big\}\subseteq\mathbb{R}^4.
$$

Set $K^{\mathrm{or}}_{ijk}:=\mathbb{R}^4\setminus M^{2,\delta,\mathrm{or}}_{ijk}$. The global oriented simplicial region is then given by 
$$
\Omega^{\mathrm{or}}_\delta
:=\prod_{i<j<k} K^{\mathrm{or}}_{ijk}
\ \subseteq   \mathbb{R}^{4\binom{N}{3}}.
$$
\end{definition}

\begin{remark}
The use of the signed coordinates $y_{ab}$ in the definition is $M^{2,\delta,\mathrm{or}}_{ijk}$ captures the orientations of the edges in relation to the triadic connection. If an edge is large in magnitude but has the opposite orientation to the triad, then $y_{ab}$ is negative and hence fails the threshold test. Thus, the single inequality $y_{ab}\geq \delta$ simultaneously enforces magnitude and orientation agreement.\\
Furthermore, the ambient space remains $\mathbb{R}^{4\binom{N}{3}}$ as in the symmetric case, since each $K^{\mathrm{or}}_{ijk}$ contains a $4$-tuple, and there are $\binom{N}{3}$ such sets. 
Because we use the signed coordinates $y_{ij},y_{ik},y_{jk}$, odd permutations of $(i,j,k)$ simply flip all three signs simultaneously, so they do not create extra dimensions. 
Hence a single $K^{\mathrm{or}}_{ijk}$ already captures both orientations of the simplex. 
\end{remark}

Now, in a similar fashion to the symmetric case, one can compute the boundary $\partial M^{2,\delta,\mathrm{or}}_{ijk}$ as the union of the following four almost-everywhere disjoint sets, 

\begin{align*}
X_1 &= \{\,|A^{(2)}_{ijk}|=\delta,\ \min(y_{ij},y_{ik},y_{jk})\leq \delta\,\}, \\
X_2 &= \{\,|A^{(2)}_{ijk}|\geq \delta,\ y_{ij}=\delta,\ y_{ik}\geq \delta,\ y_{jk}\geq \delta\,\}, \\
X_3 &= \{\,|A^{(2)}_{ijk}|\geq \delta,\ y_{ik}=\delta,\ y_{ij}\geq \delta,\ y_{jk}\geq \delta\,\},\\
X_4 &= \{\,|A^{(2)}_{ijk}|\geq \delta,\ y_{jk}=\delta,\ y_{ij}\geq \delta,\ y_{ik}\geq \delta\,\}. 
\end{align*}
From this, we can make the following definition of outward pointing.\\

\begin{definition}[Outward Pointing at Threshold $\delta$ (Oriented Case)]\label{def: outward pointing antisymmetric}
Fix a triple $(i,j,k)$ and write $u=(A^{(2)}_{ijk},y_{ij},y_{ik},y_{jk})$ with $y_{ab}:=\sigma_{ijk}A^{(1)}_{ab}$, $\sigma_{ijk}=\text{sgn}(A^{(2)}_{ijk})$.
Let $G^{(ijk)}(u)$ denote the dynamics in these coordinates. 
Set the barrier functions
$$
b_1:=\delta-|A^{(2)}_{ijk}|,\qquad b_2:=y_{ij}-\delta,\qquad b_3:=y_{ik}-\delta,\qquad b_4:=y_{jk}-\delta.
$$
We say the dynamics are \emph{outward--pointing (oriented) at threshold $\delta$} if for every boundary point $u^\ast$ of
$$
M^{2,\delta,\mathrm{or}}_{ijk}=\{\,|A^{(2)}_{ijk}|\geq \delta\ \text{ and }\ \min(y_{ij},y_{ik},y_{jk})<\delta\,\}
$$
and for every active constraint $b_\ell(u^\ast)=0$, one has
$$
\dot b_\ell(u^\ast)=\nabla b_\ell(u^\ast)\cdot G^{(ijk)}(u^\ast)\geq 0.
$$
Equivalently, at regular boundary points, this is the list of inequalities
$$
\text{on }X_1:\ \text{sgn}(A^{(2)}_{ijk})\,\dot A^{(2)}_{ijk}\leq 0,\quad
\text{on }X_2:\ \dot y_{ij}\geq 0,\quad
\text{on }X_3:\ \dot y_{ik}\geq 0,\quad
\text{on }X_4:\ \dot y_{jk}\geq 0.
$$

At singular points (several $b_\ell=0$) the condition must hold for all outward normals in the positive cone generated by the corresponding $\nabla b_\ell$.
\end{definition}

This formulation now allows us to state the main simplicial retention theorem in the oriented case. \\

\begin{theorem}[Retention of oriented simplicial structure]\label{thm: oriented retention theorem}
Fix a threshold parameter $\delta>0$. Assume $\mathcal{H}(t)$ is an adaptive triadic network dynamical system that converges to the antisymmetric regime. Suppose the vector field $(\textbf{x},A^{(1)},A^{(2)})$ is $C^1$ on $\mathbb{R}^{N+N^2+N^3}\setminus\mathcal{H}_0$, 
    (where  $\mathcal{H}_0$ is the union of vanishing hyperplanes as in Theorem \ref{thm:retention}), and  suppose the dynamics are outward--pointing (oriented) at threshold $\delta$ for every triple $(i,j,k)$, as in Definition \ref{def: outward pointing antisymmetric}. 
Let
$$
K^{\mathrm{or}}_{ijk}:=\{\,|A^{(2)}_{ijk}|<\delta\ \text{ or }\ y_{ij},y_{ik},y_{jk}\geq \delta\,\},\qquad
\Omega^{\mathrm{or}}_\delta:=\prod_{i<j<k}K^{\mathrm{or}}_{ijk}\subset\mathbb{R}^{4\binom{N}{3}}.
$$
If for some $t_1\geq t_0$ we have $(A^{(1)}(t_1),A^{(2)}(t_1))\in\Omega^{\mathrm{or}}_\delta$, then
$$
(A^{(1)}(t),A^{(2)}(t))\in\Omega^{\mathrm{or}}_\delta\quad\text{for all }t\geq t_1.
$$
\end{theorem}

\begin{proof}
Let 
$$
\tau:=\inf\{t\geq t_1:\ (A^{(1)}(t),A^{(2)}(t))\notin\Omega^{\mathrm{or}}_\delta\}.
$$
If $\tau<\infty$, then for some triple $(i,j,k)$ the quadruple $u(t)=(A^{(2)}_{ijk}(t),y_{ij}(t),y_{ik}(t),y_{jk}(t))$ satisfies $u(t)\in K^{\mathrm{or}}_{ijk}$ for $t<\tau$ and $u(\tau)\in\partial M^{2,\delta,\mathrm{or}}_{ijk}$.
Let $I:=\{\ell\in\{1,2,3,4\}: b_\ell(u(\tau))=0\}$ be the active set. By continuity and minimality of $\tau$, there exists $\eta>0$ such that $b_\ell(u(t))>0$ for all $\ell\in I$ and all $t\in(\tau-\eta,\tau)$. To enter $M^{2,\delta,\mathrm{or}}_{ijk}$ at $\tau$ one must have $\dot b_\ell(u(\tau))\leq 0$ for some $\ell\in I$. But the outward--pointing condition yields $\dot b_\ell(u(\tau))\geq 0$ for every $\ell\in I$, a contradiction.
\end{proof}

\subsection{Mixed Regime}

In the asymptotic regime classification, we can study the emergence of either unoriented or oriented simplicial complexes when our hypernetwork converges to either the symmetric or antisymmetric limit. If it converges, however, to the mixed regime limit, we lack the essential properties of symmetry to study the emergence of either type of simplicial complex. Nonetheless, there are other algebraic-topological tools one can use, if instead we shift our attention to the emergence of semi-simplicial sets (also known as $\Delta-$sets), which are the minimal categorical objects that capture these higher-order interactions and maintain techniques and tools from homology theory and computational topology.
We first discuss some category-theoretic preliminaries and show why simplicial sets are too specific to encounter problems in this setting. We are forced to preclude part of their definition, leaving semi-simplicial sets as the most suitable object in the mixed regime. \\

\begin{definition}\label{def: simplicial set (contravariant functor)}
    Let $\Delta$ be the category whose objects are finite, non-empty, totally ordered sets 
    $$[n] = \{ 0, 1, \dots , n\},$$
    for $ n \geq 0$, 
    and morphisms are the order preserving functions. that is, if $i<j$ in $[n]$, then $f(i)<f(j)$ in $[m]$ for any morphism $f\colon[n]\to[m]$. Then a \emph{simplicial set} is a contravariant functor from $\Delta$ to $\mathsf{Set}, $ ( where $\mathsf{Set}$ denotes the category of sets). More generally, for any category $\mathsf{C},$ a \emph{simplicial object} in $\mathsf{C}$ is a functor $X : \Delta^{\text{op}} \rightarrow \mathsf{C}.$
\end{definition}

We can encode the combinatorial data  of a simplicial set in a more verifiable way by leveraging the fact that the category $\Delta$ has a natural generating set of morphisms. For each $n \in \mathbb{N}_{0}$, there are $n + 1$ injections $\delta^i : [n-1] \rightarrow [n]$ called the coface maps, and $n + 1$ surjections $ \sigma^j : [ n + 1 ] \rightarrow [ n ],$ called the codegeneracies. Explicitly, these are given by

$$\delta^i(j) = \begin{cases}
    j & \text{ if } j < i \\
    j + 1 & \text{ if } j \geq i 
\end{cases} \, \, \,  \text{    and    } \, \, \,   \sigma^j(i)=
  \begin{cases}
    i, & i\le j,\\
    i-1, & i>j.
  \end{cases}$$
One sees that  $\delta^i$ is the unique, order-preserving, injective map that omits $i$, and that $\sigma^i$ is the unique, surjective, order preserving map that hits $i$ twice. Now, these two maps satisfy the following relations, called the cosimplicial relations.
\begin{align}\label{equations: cosimplicial relations}
\delta^j \circ \delta^i &= \delta^i \circ \delta^{j-1} \quad && \text{for } i < j \\
\sigma^j \circ \sigma^i &= \sigma^i \circ \sigma^{j+1} \quad && \text{for } i \le j \\
\sigma^j \circ \delta^i &= 
\begin{cases}
    \delta^i \circ \sigma^{j-1} & \text{if } i < j \\
    \mathrm{id} & \text{if } i=j \text{ or } i=j+1 \\
    \delta^{i-1} \circ \sigma^j & \text{if } i > j+1
\end{cases}
\end{align}

One can verify that every morphism of $\Delta$ can be expressed as a composition of the coface and codegeneracy maps. If $X$ is a simplicial set, we write $d_i := X(\delta^i)$ and $s_i := X(\sigma^i),$ and call these the face and degeneracy maps respectively. Due to $X$ being contravariant, these maps then satisfy dual relations to the cosimplicial relations in \ref{equations: cosimplicial relations}, called the \emph{simplicial relations}. We can consequently give an alternate definition of a simplicial set as follows. \\

\begin{definition}\label{def: simplicial set (sequences of sets definition)}
    A \emph{simplicial set} $X$ is a collection of sets $X_n$ for each $n \in \mathbb{N}_0$ with functions $d_i : X_n \rightarrow X_{n-1}$ and $s_i : X_n \rightarrow X_{n+1}$ for all $0 \leq i \leq n$ and for each $n$ satisfying the following relations, 
    $$\begin{aligned}
d_i d_j &= d_{j-1} d_i, && i<j,\\
s_i s_j &= s_{j+1} s_i, && i\le j,\\[6pt]
d_i s_j &=
  \begin{cases}
      1, & i=j \text{ or } i=j+1,\\
      s_{j-1} d_i, & i<j,\\
      s_j d_{i-1}, & i>j+1.
  \end{cases}
\end{aligned}$$
\end{definition}

Given an adaptive triadic system $(\mathbf{x},A^{(1)}, A^{(2)})$, for each $n \in \mathbb{N}_0$, we define sets $X_n$ as follows, 
\begin{equation}\label{eq:X_definitions}
\begin{aligned}
X_0 &:= X(\{0\}) = \{1, \dots, N\},\\
X_1 &:= X(\{0,1\}) = \left\{ (i,j) \in [N]^2 : \left|A^{(1)}_{ij}\right| \geq \delta \right\},\\
X_2 &:= X(\{0,1,2\}) = \left\{ (i,j,k) \in [N]^3 : \left|A^{(2)}_{ijk}\right| \geq \delta \right\}.
\end{aligned}
\end{equation}

and $X_n := X([n]) = \emptyset$ for all $n > 2.$ We now see that if we attempt to define degeneracy maps for all $n \geq 0$, we run into a problem when $n = 3$, because unless $X_2$ is empty, which it need not necessarily be, there is no map  $X_2 \rightarrow \emptyset $.  Consequently, the degeneracy maps of a simplicial set cause problems and do not allow us to model a simplicial structure emerging on the hypergraph. If we therefore remove the degeneracy maps, we are instead left with an object known as a semi-simplicial set, or $\Delta-$set.\\

\begin{definition}[Semi-simplicial set or $\Delta$–set]\label{def:DeltaSet}
A \emph{semi-simplicial set}, or \emph{$\Delta$–set} $X$ is a collection of sets $X_n$ for each integer $n\ge 0$
together with \emph{face maps}
$$
d_i\colon X_n \longrightarrow X_{n-1},\qquad 0\leq i\leq n,
$$
such that for every $n \geq 1$ they satisfy the simplicial relations
$$
d_i d_j \;=\; d_{j-1} d_i ,\qquad i<j\le n.
$$

(Equivalently, $X$ is a contravariant functor $X : \Delta_{ \mathsf{inj}}^{\mathrm{op}}\rightarrow\mathsf{Set}$,
where $\Delta_{\mathsf{inj}}\subseteq\Delta$ is the sub-category whose morphisms are the
\emph{injective} order–preserving maps.)\\
\end{definition}

\begin{remark}\leavevmode
\begin{enumerate}
\item
Compared with \ref{def: simplicial set (sequences of sets definition)}, a simplicial set has both face maps $d_i$
and degeneracy maps $s_i$, subject to a larger family of identities.
A semi-simplicial set is therefore a strictly weaker or more general object:
every simplicial set becomes a $\Delta$–set by forgetting the $s_i$,
but not every $\Delta$–set admits degeneracies that would upgrade it to a simplicial set.

\item For most purposes in homology or persistent-homology, one loses no essential
information by dropping the degeneracies, since chain complexes only use face maps.

\end{enumerate}
\end{remark}

The following theorem shows that our  hypergraph is a semi-simplicial set if it satisfies the downward facing closure condition.\\

\begin{theorem}
    Let $\delta > 0,$ and suppose $\mathcal{H}(t)$ is a triadic, adaptive network dynamical system encoded by the weighted adjacency matrices $A^{(1)}$ and $A^{(2)}$. Define sets $X_n$ as in \eqref{eq:X_definitions}.

Define face maps by 
$$d_0,d_1 : X_1 \rightarrow X_0, \, \, 
d_0(i,j)=j, \, \,  d_1(i,j)=i, $$
$$d_0,d_1,d_2 : X_2 \rightarrow X_1,
d_0(i,j,k)=(j,k),\;
d_1(i,j,k)=(i,k),\;
d_2(i,j,k)=(i,j).$$

For $n\ge 3$ put $d_i=\varnothing:\varnothing\to\varnothing$.  Then the sets  and face maps $\left\{ X_n , d_i \right\}$ form a semi-simplicial set (at time t, and with parameter $\delta > 0$), if and only if, for all $i,j,k \in  \{ 1 , \dots , N\}$, 

$$
|A^{(2)}_{ijk}| \geq \delta \Longrightarrow \, \, 
|A^{(1)}_{ij}|,
\,|A^{(1)}_{ik}|,
\,|A^{(1)}_{jk}|
\, \geq\ \, \delta.
$$
\end{theorem}

\begin{proof}
($\Rightarrow$)  
 First, suppose that the sets and face maps derived from $\mathcal{H}$ forms a semi-simplicial set. It then follows that  the face maps are well defined, so $d_i(X_2)\subseteq X_1$ for $i=0,1,2$.
Take $(i,j,k)\in X_2$; then $|A^{(2)}_{ijk}|\ge\delta$.
Since $d_2(i,j,k)=(i,j)\in X_1$, we have $|A^{(1)}_{ij}|\ge\delta$,
and similarly for $(i,k)$ and $(j,k)$.  Hence, we have that $$
|A^{(2)}_{ijk}| \geq \delta \Longrightarrow \, \, 
|A^{(1)}_{ij}|,
\,|A^{(1)}_{ik}|,
\,|A^{(1)}_{jk}|
\, \geq\ \, \delta.
$$

($\Leftarrow$)  
Conversely, suppose we have that for all $i,j,k \in \{ 1, \dots , N\},$ we have that $
|A^{(2)}_{ijk}| \geq \delta $ implies that $ 
|A^{(1)}_{ij}|, \,|A^{(1)}_{ik}|, \,|A^{(1)}_{jk}| $ are all greater than or equal to $\delta.$ Then every boundary edge of a triangle in $X_2$ lies in $X_1$, so
$d_0,d_1,d_2$ land in $X_1$.  All higher $X_n$ are empty; thus, all other
face maps are the unique maps $\varnothing\to\varnothing$.
The only non-vacuous simplicial identity is $d_0d_1=d_0d_0$ on $X_2$,
and a direct check shows both sides send $(i,j,k)$ to $k$.
Therefore  the sets  and face maps $\left\{ X_n , d_i \right\}$ form a semi-simplicial set.
\end{proof}

Just as in symmetric and antisymmetric cases, we follow a similar procedure as before and produce the bad set where a simplicial complex structure breaks down. \\

\begin{definition}[Semi--simplicial region at threshold $\delta$]
For each ordered triple $(i,j,k)$ with $i,j,k$ pairwise distinct, define the bad set
$$
M^{2,\delta,\mathrm{ss}}_{ijk}
:=\Big\{(A^{(2)}_{ijk},A^{(1)}_{ij},A^{(1)}_{ik},A^{(1)}_{jk}) : |A^{(2)}_{ijk}|\geq \delta \ \text{ and }\ \min\big(A^{(1)}_{ij},A^{(1)}_{ik},A^{(1)}_{jk}\big)<\delta\Big\}\subseteq \mathbb{R}^4.
$$
  
Set
$$
K^{\mathrm{ss}}_{ijk}:=\mathbb{R}^4\setminus M^{2,\delta,\mathrm{ss}}_{ijk},
$$
and define the global semi-simplicial region by the product over all ordered triples without repetition
$$
\Omega^{\mathrm{ss}}_\delta
:=\prod_{(i,j,k)\in[N]_{i \neq j \neq k}^3} K^{\mathrm{ss}}_{ijk}
\ \subseteq\ \mathbb{R}^{\,4\,N(N-1)(N-2)} \;=\; \mathbb{R}^{\,4\,\frac{N!}{(N-3)!}}.
$$
\end{definition}

\begin{remark}
In the semi-simplicial setting there is no relationship between $A^{(2)}_{ijk}$ and any $S_3$ permutation of its indices, so each ordered triple
contributes an independent $4$--tuple. There are $N \cdot (N-1) \cdot (N-2)$ such tuples. Consequently, the ambient space is $\mathbb{R}^{\,4\,N(N-1)(N-2)}$.
\end{remark}

In a similar fashion to the symmetric and antisymmetric regimes from earlier, the boundary of $M^{2,\delta,\mathrm{ss}}_{ijk}$  decomposes (up to measure–zero overlaps) into four pieces $X_1,X_2, X_3, X_4$:
\begin{align*}
X_1 &:= \big\{|A^{(2)}_{ijk}|=\delta,\ \min\{|A^{(1)}_{ij}|,|A^{(1)}_{ik}|,|A^{(1)}_{jk}|\}\leq\delta\big\},\\
X_2 &:= \big\{|A^{(2)}_{ijk}|\geq\delta,\ |A^{(1)}_{ij}|=\delta,\ |A^{(1)}_{ik}|\geq\delta,\ |A^{(1)}_{jk}|\geq\delta\big\},\\
X_3 &:= \big\{|A^{(2)}_{ijk}|\geq\delta,\ |A^{(1)}_{ik}|=\delta,\ |A^{(1)}_{ij}|\geq\delta,\ |A^{(1)}_{jk}|\geq\delta\big\},\\
X_4 &:= \big\{|A^{(2)}_{ijk}|\geq\delta,\ |A^{(1)}_{jk}|=\delta,\ |A^{(1)}_{ij}|\geq\delta,\ |A^{(1)}_{ik}|\geq\delta\big\}.
\end{align*}
On the regular points of $X_\ell$, the outward unit normal $n : \partial M_{ijk}^{2,\delta,ss} \rightarrow \mathbb{R}^4,$ can be chosen as

$$n\bigl(A^{(2)}_{ijk},A^{(1)}_{ij},A^{(1)}_{ik},A^{(1)}_{jk}\bigr)  = \begin{cases} 
    (-\text{sgn}(A^{(2)}_{ijk}), 0, 0, 0) & \text{ on } X_1 \\ 
    \bigl(0, \text{sgn}(A^{(1)}_{ij}), 0, 0\bigr) & \text{ on } X_2\\
    \bigl(0, 0, \text{sgn}(A^{(1)}_{ik}), 0\bigr) & \text{ on } X_3\\
     \bigl(0, 0, 0, \text{sgn}(A^{(1)}_{jk})\bigr) & \text{ on } X_4
\end{cases}$$

At singular points $x^\ast\in X_p\cap X_q$ we take the outward normal cone
$$
N_{\partial M}(x^\ast):=\Big\{\,\sum_{r\in I(x^\ast)}\lambda_r\,n_r(x^\ast)\ \Big|\ \lambda_r\in\mathbb{R}_{\geq 0}\Big\},
\quad I(x^\ast):=\{\ell:\ x^\ast\in X_\ell\}.
$$

\begin{definition}[Outward–pointing at threshold $\delta$ for the mixed regime]\label{def: outward pointing semi simplicial}
Let $F^{(ijk)}:\mathbb{R}^4\to\mathbb{R}^4$ denote the vector field of the reduced $(i,j,k)$–subsystem
$$
F^{(ijk)}(x)=\big(\dot A^{(2)}_{ijk},\,\dot A^{(1)}_{ij},\,\dot A^{(1)}_{ik},\,\dot A^{(1)}_{jk}\big).
$$
We say the hypergraph is \emph{outward–pointing at threshold $\delta$} if for every triple $(i,j,k)$ and every $x\in\partial M^{2,\delta}_{ijk}$ we have
$$
n\cdot F^{(ijk)}(x)\ \geq\ 0 \quad \text{for all } n\in N_{\partial M}(x).
$$
Equivalently, it suffices to verify the following sign–derivative inequalities:
\begin{align*}
&\text{on }X_1:\qquad \text{sgn}(A^{(2)}_{ijk})\,\dot A^{(2)}_{ijk}\ \leq\ 0,\\
&\text{on }X_2:\qquad \text{sgn}(A^{(1)}_{ij})\,\dot A^{(1)}_{ij}\ \geq\ 0,\\
&\text{on }X_3:\qquad \text{sgn}(A^{(1)}_{ik})\,\dot A^{(1)}_{ik}\ \geq\ 0,\\
&\text{on }X_4:\qquad \text{sgn}(A^{(1)}_{jk})\,\dot A^{(1)}_{jk}\ \geq\ 0.
\end{align*}
\end{definition}

This allows us to state our semi-simplicial set retention theorem.\\

\begin{theorem}[Retention of Semi-Simplicial Structure]\label{thm: mixed regime retention main theorem}
Fix a simplicial tolerance parameter $\delta > 0$.  
Assume $H(t)$ is an adaptive triadic network dynamical system and consider the semi--simplicial regime with no symmetry identifications.  Suppose the vector field $(\textbf{x}, A^{(1)}, A^{(2)})$ is $C^1$ on $\mathbb{R}^{N+N^2+N^3}\setminus\mathcal{H}_0$, 
    (where  $\mathcal{H}_0$ is the union of vanishing hyperplanes as in Theorem \ref{thm:retention}), and suppose the system is outward--pointing at every ordered triple $(i,j,k)$ with threshold $\delta$ in the sense of Definition \ref{def: outward pointing semi simplicial}.  
If at some time $t_1 \geq 0$ the configuration $(A^{(1)}(t_1),A^{(2)}(t_1))$ lies in $\Omega^{\mathrm{ss}}_\delta$, then
$$
(A^{(1)}(t),A^{(2)}(t)) \in \Omega^{\mathrm{ss}}_\delta \quad \text{for all } t \geq t_1,
$$
i.e.\ once semi-simplicial at threshold $\delta$, the configuration remains semi-simplicial for all future time.
\end{theorem}

\begin{proof}
The argument follows identically to the symmetric and oriented cases.  
Consider the first-exit time from $\Omega^{\mathrm{ss}}_{\delta} = \prod_{(i,j,k)} K^{\mathrm{ss}}_{ijk}$ into some bad set $M^{2,\delta,\mathrm{ss}}_{ijk} \subset \mathbb{R}^4$.  
As before, an exit at time $\tau$ would require the local vector field 
$$
F^{(ijk)} := \big( \dot{A}^{(2)}_{ijk},\, \dot{A}^{(1)}_{ij},\, \dot{A}^{(1)}_{ik},\, \dot{A}^{(1)}_{jk} \big)
$$
to point strictly inward, i.e. $n \cdot F^{(ijk)} < 0$ for some outward normal $n$ on $\partial M^{2,\delta,\mathrm{ss}}_{ijk}$.  
The outward-pointing condition,  however, precludes this, hence no exit is possible and $(A^{(1)}(t),A^{(2)}(t)) \in \Omega^{\mathrm{ss}}_{\delta}$ for all $t \geq t_1$.
\end{proof}

%% file: Conclusion.tex
\section{Conclusion and Outlook}

We have developed a tensor based framework for adaptive higher–order network dynamics on directed hypergraphs, guided by a generalised Kuramoto model as a prototype. By combining the representation theory of the symmetric group with Frobenius norm order parameters, we identified three asymptotic symmetry regimes, symmetric, antisymmetric, and mixed, and showed that each yields a distinct stable combinatorial object: unoriented, oriented, or semi-simplicial. The resulting theorems provide the first rigorous conditions for simplicial complex emergence in adaptive triadic systems.

The framework generalises naturally to higher interaction orders. The following generalised adaptive Kuramoto model serves as a prototypical example:
\[
\dot{\theta}_i = \omega_i + 
\sum_{m=1}^{K} \frac{1}{N^m} 
\sum_{i_1,\dots,i_m=1}^{N} 
A^{(m)}_{ii_1\dots i_m}\, 
f_m(\theta_i,\theta_{i_1},\dots,\theta_{i_m}),
\]
where each interaction tensor $A^{(m)}(t)$ evolves according to a local adaptation law $\dot{A}^{(m)} = g^{(m)}(A^{(m)},\theta)$.
Each tensor can be decomposed into isotypic components under the action of $S_{m+1}$, and the same local boundary and drift criteria can be formulated to detect simplicial emergence at order $m$. Although the combinatorial complexity grows rapidly with $m$, the representation-theoretic machinery remains formally identical. Several directions emerge for future work:
\begin{enumerate}
    \item \textbf{Topological analysis.} With rigorous conditions ensuring simplicial structure, one can now compute homological invariants (e.g.\ Betti numbers or persistent homology) for adaptive higher--order systems, tracking the creation and destruction of higher--dimensional features.
    \item \textbf{Analytical limits.} Studying the dense large-$N$ limit will link the discrete tensor formalism to continuum objects such as hypergraphons, yielding coupled mean-field equations for the evolving kernels and node densities.
    \item \textbf{Applications and data.} The framework could be applied to empirical adaptive systems-such as neural ensembles, multi-agent coordination, or group contagion-where higher-order structure co-evolves with dynamics.
\end{enumerate}

%% file: acknowledgments.tex
\section*{Code Availability}
The Python code used to generate all simulations and figures in this paper is publicly available on GitHub. The code relies on standard libraries and can be found at
\url{https://github.com/fergal-murphy/simplicial_emergence_hypergraphs}

\section*{Acknowledgments}
   This work was supported by the European Union’s Horizon Europe
   Marie Skłodowska-Curie Actions under the \say{BeyondTheEdge: Higher-Order
   Networks and Dynamics} project (Grant Agreement No. 101120085).